\documentclass[11pt]{amsart}
\usepackage{amsthm}

\usepackage[all]{xy}
\usepackage{amssymb}
\usepackage{enumerate}
\usepackage{mathrsfs}
\usepackage{epsfig}
\usepackage{graphicx}
\usepackage{subfig}
\usepackage{float}
\usepackage{epigraph}

\numberwithin{equation}{section}

\newtheorem{thm}{Theorem}[section]

\newtheorem{lem}[thm]{Lemma}
\newtheorem{prop}[thm]{Proposition}
\newtheorem{rem}{Remark}[section]
\newtheorem{example}[thm]{Example}

\newcommand{\eq}[1]{(\ref{#1})}

\renewcommand{\Re}{\operatorname{\rm Re}}
\renewcommand{\Im}{\operatorname{\rm Im}}

\newcommand{\beqast}{\begin{eqnarray*}}
\newcommand{\eqast}{\end{eqnarray*}}
\newcommand{\beqa}{\begin{eqnarray}}
\newcommand{\eqa}{\end{eqnarray}}

\newcommand{\bbe}{\begin{equation}}
\newcommand{\ee}{\end{equation}}

\renewcommand{\Re}{\operatorname{\rm Re}}
\renewcommand{\Im}{\operatorname{\rm Im}}

\newcommand{\bC}{{\mathbb C}}
\newcommand{\bE}{{\mathbb E}}

\newcommand{\bQ}{{\mathbb Q}}

\newcommand{\bR}{{\mathbb R}}

\newcommand{\bZ}{{\mathbb Z}}

\newcommand{\cB}{{\mathcal B}}

\newcommand{\cF}{{\mathcal F}}

\newcommand{\cD}{{\mathcal D}}
\newcommand{\cE}{{\mathcal E}}

\newcommand{\cL}{{\mathcal L}}

\newcommand{\cC}{{\mathcal C}}

\newcommand{\cU}{{\mathcal U}}

\newcommand{\barX}{{\bar X}}
\newcommand{\uX}{{\underline X}}

\newcommand{\cEq}{{\mathcal E_q}}
\newcommand{\cEpq}{{\mathcal E^+_q}}
\newcommand{\cEmq}{{\mathcal E^-_q}}

\newcommand{\phipq}{{\phi^+_q}}
\newcommand{\phimq}{{\phi^-_q}}

\newcommand{\tV}{{\tilde V}}

\newcommand{\hG}{{\hat G}}

\newcommand{\Om}{{\Omega}}

\newcommand{\be}{\beta}
\newcommand{\De}{\Delta}
\newcommand{\de}{\delta}
\newcommand{\eps}{\epsilon}

\newcommand{\lp}{\lambda_+}
\newcommand{\lm}{\lambda_-}
\newcommand{\La}{\Lambda}
\newcommand{\mum}{\mu_-}
\newcommand{\mup}{\mu_+}
\newcommand{\mumpr}{\mu'_-}
\newcommand{\muppr}{\mu'_+}

\newcommand{\sg}{\sigma}

\newcommand{\om}{\omega}
\newcommand{\omm}{\om_-}
\newcommand{\omp}{\om_+}

\newcommand{\ze}{\zeta}

\newcommand{\ga}{\gamma}
\newcommand{\gap}{\gamma_+}
\newcommand{\gam}{\gamma_-}

\newcommand{\barDe}{\bar\Delta}

\newcommand{\bfo}{{\bf 1}}

\begin{document}

\title[Efficient  inverse $Z$-transform and pricing barrier and lookback options]
{Efficient inverse $Z$-transform and pricing barrier and lookback options with discrete monitoring. II}
\author[
Svetlana Boyarchenko and
Sergei Levendorski\u{i}]
{
Svetlana Boyarchenko and
Sergei Levendorski\u{i}}

\begin{abstract}
We prove a simple general formula for the expectations of a function of a random walk and its running extremum,
which is more convenient for applications than general formulas in the first version of the paper. The derivation of
explicit formulas in applications significantly simplifies.
Under additional conditions, we derive analytical formulas using the inverse $Z$-transform,  Fourier/Laplace inversion and Wiener-Hopf factorization,
and discuss efficient numerical methods for realization of these formulas. As applications, the cumulative probability distribution
function of the process and its running maximum and the price of the option to exchange the maximum of a stock price for a power of 
the price
are calculated. The most efficient numerical methods use a new efficient numerical realization of the inverse $Z$-transform,  sinh-acceleration technique and simplified trapezoid rule. The program in Matlab running on a Mac with moderate characteristics achieves the precision E-10 and better in several dozen of milliseconds, and E-14 - 
in a fraction of a isecond.

\end{abstract}

\thanks{
\emph{S.B.:} Department of Economics, The
University of Texas at Austin, 2225 Speedway Stop C3100, Austin,
TX 78712--0301, {\tt sboyarch@utexas.edu} \\
\emph{S.L.:}
Calico Science Consulting. Austin, TX.
 Email address: {\tt
levendorskii@gmail.com}}

\maketitle



\tableofcontents

\section{Introduction}\label{s:intro} Let $Y, Y_j, j=1,2\ldots, $ be i.i.d. $\bR$-valued random variables on a probability space $(\Om,\cB, \bQ)$, and $\bE$ the expectation operator under $\bQ$. 
For $x\in \bR$, $X_n=x+Y_1+\cdots + Y_n, n=0,1,2,\ldots,$ is a random walk on $\bR$ starting at $x$. In applications to finance, typically, $Y$ is an increment of a L\'evy process, and the random walk appears implicitly when either a continuous time L\'evy model is approximated or options with discrete monitoring are priced.
  In the present paper,
we derive a general formula and  efficient numerical procedure for evaluation of expectations of a random walk and its extremum. 
The formula and procedure can be applied to lookback, single barrier options and single barrier options with lookback features.
The method of the paper can be used as the main basic block to price double-barrier options with lookback features and discrete monitoring, and American options with barrier/lookback features.

Let  $\barX_n=\max_{0\le m\le n}X_m$ and $\uX_n=\min_{0\le m\le t}X_m$ be the supremum
and infimum processes (defined path-wise, a.s.); $X_0=\barX_0=\uX_0=0$. For a measurable function $f$, consider 
$
V(f;n;x_1,x_2)=\bE[f(x_1+X_n, \max\{x_2, x_1+\barX_n\})]. 
$
At the first step, as in \cite{FusaiAbrahamsSguarra06}, where barrier options with discrete monitoring in the Brownian motion model  are priced, we make the discrete Laplace transform ($Z$-transform) of the series $\vec{V}:=\{V_n\}_{n=0}^\infty:=\{V(f;n;x_1,x_2)\}_{n=0}^\infty$. 
For our purposes,
it is convenient to use the equivalent transformation
\bbe\label{Vnze}
\tV(z)=\sum_{n=0}^{+\infty}z^nV_n.
\ee 
If $f$ is uniformly bounded, the series $\vec{V}$ is uniformly bounded as well, hence, $\tV(z)$ is analytic in the open unit disc, and
$\vec{V}$ can be recovered using the residue theorem:
for any $R<1$,
\bbe\label{izeT0}
V_n=\frac{1}{2\pi i}\int_{|z|=R}z^{-n-1}\tV(z)dz.
\ee
The standard and popular approximation to the integral on the RHS is the trapezoid rule. However, if $n$ is very large,
then the trapezoid rule becomes very inefficient as we discuss in Sect. \ref{s:eff_inverse_ze} and illustrate with  numerical examples in
Sect. \ref{s:numer}. The first contribution of the paper is a new efficient method for a numerical evaluation the integral
on the RHS of \eq{izeT0}. The idea is to
deform the contour of integration $\{z=Re^{i\varphi}\ |\ -\pi<\varphi<\pi\}$ into a contour of the form
$\cL_{L; \sg_\ell,b_\ell,\om_\ell}=\chi_{L; \sg_\ell,b_\ell,\om_\ell}(\bR)$, where the conformal map $\chi_{L; \sg_\ell,b_\ell,\om_\ell}$
is defined by
 \bbe\label{eq:sinhLapl}
\chi_{L; \sg_\ell,b_\ell,\om_\ell}(y)=\sg_\ell +i b_\ell\sinh(i\om_\ell+y),
\ee
 $b_\ell>0$, $\sg_\ell\in\bR$ and $\om_\ell\in (-\pi/2,\pi/2)$. The deformation is possible under natural conditions on
 the domain of analyticity of $\tV$; these conditions are satisfied in applications that we consider.
 After the deformation and  corresponding change of variables, the simplified trapezoid rule is applied. The resulting procedure is faster and more accurate than the trapezoid rule.
 We hope that a new efficient numerical method for the evaluation of the inverse $Z$-transform \eq{izeT0}
  is of a general interest.

  The second contribution of the paper is a general formula for $\tV(f;z;x_1,x_2)$ in terms of the expected present value operators
  (EPV-operators) $\cE^\pm_q$ under the supremum and infimum processes introduced in \cite{NG-MBS,EPV,IDUU}. 
  The formula and its proof are essentially identical to the ones in \cite{EfficientLevyExtremum} for  L\'evy processes, only the definitions of the operators $\cE^\pm_q$ change. In the case of random walks, the action of $\cE^\pm_q$ is defined as follows. 
  For $q\in (0,1)$, let $T_q$ be a random variable with the distribution $\bE[T_q=n]=(1-q)q^n$, independent of $X$, and let $u$ be a bounded measurable function. Then
  $\cEpq u(x)=\bE[u(x+\barX_{T_q})]$ and $\cEmq u(x)=\bE[u(x+\uX_{T_q})]$. The formula is in Sect. \ref{ss:main}.
   In applications, the payoff function $f$ may increase exponentially at infinity. Hence, in order that the expectation be finite,
  one or even two tails of the probability distribution of $Y$ must decay exponentially at infinity. We formulate and prove a general theorem for
  the case of exponentially increasing payoff functions. 
  
  In Section \ref{s:evalFTposq01}, we use the Fourier transform and the equalities $\cE^\pm_q e^{ix\xi}=\phi^\pm_q(\xi) e^{ix\xi}$, where $\phi^\pm_q(\xi)$ are the Wiener-Hopf factors,
 to realize the formula derived in Section \ref{ss:main} as a sum of 1D-3D integrals; formulas for the Wiener-Hopf factors are in Sect. \ref{ss:WHF}. 
 As applications of the general theorems, in Section \ref{s:two_examples}, we derive
explicit formulas for the cumulative distribution function (cpdf) of random walk and its maximum,
and for the option to exchange $e^{\barX_T}$ for a power $e^{\be X_T}$.

  If one of the tails of the pdf of $Y$ exponentially decays at infinity,
  the characteristic function $\Phi(\xi)=\bE[e^{i\xi Y}]$ of $Y$ and the Wiener-Hopf factors admit analytic continuation to a strip around or adjacent to the real axis.
  This property allows one to use the useful property of the infinite trapezoid rule, namely, the exponential decay of the discretization error
  as a function of $1/\ze$, where $\ze$ is the step of the infinite trapezoid rule. However, in many cases of interest such as pricing options with daily monitoring and/or L\'evy processes close to the Variance Gamma process, the integrand decays too slowly at infinity, therefore, the number of terms in the simplified trapezoid rule necessary to satisfy even a moderate error tolerance can be huge. Fortunately, in all popular models, $\Phi(\xi)$ admits analytic continuation
  to a cone around the real axis and exponentially decays as $\xi\to \infty$ in the cone (the only exception is the Variance Gamma model;
  the rate of decay is a polynomial one). See \cite{EfficientAmenable} for the explicit calculation of the coni of analyticity in  popular models. Therefore, the sinh-acceleration technique used in \cite{SINHregular} 
to price European options and applied in \cite{Contrarian,ConfAccelerationStable,BSINH} to price barrier options and  evaluate special functions and the coefficients in BPROJ method respectively can be applied to greatly decrease the sizes of grids and the CPU time needed to satisfy the desired error tolerance.  The changes of variables must be in a certain agreement as in \cite{paraLaplace,paired,EfficientLevyExtremum}. 
Note that the deformation \eq{eq:sinhLapl} and the corresponding change of variables constitute an example of the application of the sinh-acceleration technique.    We show that, in some cases, one of the integrals (either outer or inner one) has to be calculated using a less efficient
family of sub-polynomial deformations introduced and used in \cite{ConfAccelerationStable}.
Numerical examples are in Section \ref{s:numer}. We demonstrate that the method based on the sinh-acceleration for the
inverse $Z$-transform can achieve the accuracy of the order of E-14 and better using Matlab and Mac with moderate characteristics, in a second or fraction of a second, and the precision of the order of 
E-10 in 20-30 msec., for options of maturity in the range $T=0.25-15Y$. In all cases, the arrays are of a moderate size. In particular, the
number of points used for the $Z$-transform inversion is several dozens in all cases.
If the trapezoid rule is used, the size of arrays and CPU time increase with the maturity, and, for maturity $T=15$,
approximately 3,000 points are needed, and the CPU time is several times larger. We also compare the results in
the case of the continuous monitoring using the methods developed in \cite{EfficientLevyExtremum} and demonstrate that in the case of daily monitoring, the relative differences are rather small even for $T=15Y$.

There is a huge body of the literature devoted to pricing options with barrier and/or lookback features, and a number of different methods
have been applied.  The methods that are conceptually close to the method of the paper are the ones that use the fast inverse Fourier transform, fast convolution or fast Hilbert transform.
In Section \ref{s:concl}, we  review several popular methods and  explain why these methods are computationally more expensive than
the method of the present paper and cannot achieve the precision demonstrated in  Section \ref{s:numer}. 
We also summarize the results of the paper and outline several extensions of the method of the paper. 
We relegate to Appendix \ref{s:tech} several technicalities. Figures and   tables are in Appendix
\ref{ss:figures}.

\section{Efficient inverse $Z$-transform}\label{s:eff_inverse_ze}

\subsection{Trapezoid rule}
Let a sequence $\vec V=(V)_{n=0}^{\infty} $ and $A>0$ satisfy\footnote{In applications to pricing options in an exponential L\'evy model with the characteristic exponent $\psi$, $A=e^{\bar \De \psi(-i\be)}$, where $\bar\De$ is the time step,
 and $\be\in \bR$ depends on the option's payoff.} 
\bbe\label{condconvze}
H(\vec V,A):=\sum_{n=0}^\infty |V_n|A^{n}<+\infty.
\ee
Then, for any $z\in \cD(0,A):=\{z\in \bC\ |\ |z|\le A\}$, the series \eq{Vnze} 
converges and defines the function analytic in $\cD(0,A)$ (meaning: analytic in the open domain 
$\{|z|<1/A\}$ and continuous up to the boundary)\footnote{Recall that the function $\tilde V(1/z)$ is called the $Z$-transform of the series $\vec V$.}. Hence, $V_n$ can be recovered using the Cauchy residue theorem. Explicitly,
for any $R<A$, \eq{izeT0} holds. Changing the variable $z\mapsto z R$, and introducing $h(z)(=h(R,z))=(z R)^{-n}\tV(z R)$, we obtain
\bbe\label{izeT}
V_n=\frac{1}{2\pi i}\int_{|z|=1}h(z)\frac{dz}{z},\ n=0,1,2,\ldots
\ee 
Usually,
one evaluates  the RHS of \eq{izeT}, denote it $I(h)$, using the  trapezoid rule: 
\bbe\label{defTM}
T_M(h) = (1/M) \sum_{k=0}^{M-1} h(\zeta_M^k),
\ee
where $M>1$ is an integer, and $\zeta_M = \exp(2\pi i/M)$ is the standard primitive $M$-th root of unity.
For $0<a<b$, denote   $\cD_{(a,b)}:=\{z\ |\ a<|z|<b\}$. Since $R<A$, $h(z)$ is analytic in the annulus $\cD_{(1/\rho,\rho)}$.
The Hardy norm of $h$ is
\[
\|h\|_{\cD_{(1/\rho,\rho)}}=\frac{1}{2\pi i}\int_{|z|=1/\rho}|h(z)|\frac{dz}{z}+\frac{1}{2\pi i}\int_{|z|=\rho}|h(z)|\frac{dz}{z}.
\]
The error bound is well-known; for completeness, we give the proof in Sect. \ref{ss:proof_disctraperror}.
\begin{thm}\label{disctraperror}
Let $h$ be analytic in  $\cD_{(1/\rho,\rho)}$, where $\rho>1$.  
The error of the trapezoid approximation admits the bound
\bbe\label{errtrapgen}
|T_M(h)-I(h)|\le \frac{\rho^{-M}}{1-\rho^{-M}}\|h\|_{\cD_{(1/\rho,\rho)}}.
\ee
\end{thm}
If $V_n$ are real, then $\overline{h(z)}=h(\bar z)$, hence, we can choose an odd  $M=2M_0+1$ 
and obtain
\bbe\label{defTMsym}
T_M(h) = (2/M) \Re\sum_{k=0}^{M-1} h(\zeta_M^k)(1-\de_{k0}/2).
\ee 
\subsection{Sinh-acceleration}\label{ss:sinh_accel_ze} Let there exist $\ga\in (0,\pi)$ such that
$\tV$ admits
 analytic continuation to a domain of the form $\cU(R,\rho,\ga)=(\cD(R/\rho,R\rho)-(\cC_\ga\cup\{0\}))\setminus \cD(0,R/\rho)$, where 
 $\cC_\ga=\{z\ |\ \mathrm{arg}\,z\in (-\ga,\ga)\}$, and let
there exist $C_{\tV}>0$ and $a_{\tV}<n$ such that
\bbe\label{bound_tV_disc}
|\tV(z)|\le C_{\tV}|z|^{a_{\tV}},\ z\in \cU(R,\rho,\ga).
\ee
Then we can deform the contour of integration $\{z=Re^{i\varphi}\ |\ -\pi<\varphi<\pi\}$ in \eq{izeT0} into a contour of the form
$\cL_{L; \sg_\ell,b_\ell,\om_\ell}=\chi_{L; \sg_\ell,b_\ell,\om_\ell}(\bR)$, where the conformal map $\chi_{L; \sg_\ell,b_\ell,\om_\ell}$
is defined by \eq{eq:sinhLapl}.  After the transformation, we make the corresponding change of variables and reduce to the integral over $\bR$:
 \bbe\label{izeT0sinh}
V_n=\int_{\bR}\frac{b_\ell}{2\pi }\chi_{L; \sg_\ell,b_\ell,\om_\ell}(y)^{-n-1}\cosh(i\om_\ell+y)\tV(\chi_{L; \sg_\ell,b_\ell,\om_\ell}(y))dy,
\ee
denote by $f_n(y)$ be the integrand on the RHS of \eq{izeT0sinh}, and apply the infinite trapezoid rule
\bbe\label{Vn_inf_sinh}
V_n\approx \ze_\ell \sum_{j\in \bZ}f_n(j\ze_\ell).
\ee
An error bound
is easy to derive because 
 the function $f_n$  is analytic in a strip $S_{(-d,d)}:=\{\xi\ | \Im\xi\in (-d,d)\}$, where $d>0$ depends on the domain of analyticity of $\tV$ and the choice of the parameters $\sg_\ell, \om_\ell, b_\ell$, and  \eq{bound_tV_disc} holds. With an appropriate choice of the parameters $\om_\ell, \sg_\ell, b_\ell$ and $d$,
$\lim_{R\to \pm\infty}\int_{-d}^d |f_n(i s+R)|ds=0,$
and 
\bbe\label{Hnorm}
H(f_n,d):=\|f_n\|_{H^1(S_{(-d,d)})}:=\lim_{s\downarrow -d}\int_\bR|g(i s+ t)|dt+\lim_{s\uparrow d}\int_\bR|g(i s+t)|dt<\infty.
\ee
We write $f_n\in H^1(S_{(-d,d)})$. The following key lemma is proved in \cite{stenger-book} using the heavy machinery of sinc-functions. A simple proof (analogous to the proof of Theorem \ref{disctraperror}) can be found in
\cite{paraHeston}.
\begin{lem}[\cite{stenger-book}, Thm.3.2.1] For $f_n\in H^1(S_{(-d,d)})$,
the error of the infinite trapezoid rule admits an upper bound 
\bbe\label{Err_inf_trap}
{\rm Err}_{\rm disc}\le H(f_n,d)\frac{\exp[-2\pi d/\ze]}{1-\exp[-2\pi d/\ze]}.
\ee
\end{lem}
 Once
an approximate bound $H_{\mathrm{appr.}}(f_n,d)$ for $H(f_n,d)$ is derived, it becomes possible to satisfy the desired error tolerance
with a good accuracy letting
\bbe\label{rec_ze_ze}
\ze_\ell=2\pi d\ln(H_{\mathrm{appr.}}(f_n,d)/\eps).
\ee
Since $f_n(y)$ decays as $((b/2)e^{|y|})^{-n-1}$ as $y\to\pm \infty$, it is straightforward to choose the truncation of the infinite sum on the RHS of \eq{Vn_inf_sinh}:
\bbe\label{Vn_inf_sinh_trunc}
V_n\approx \ze_\ell \sum_{|j|\le M_0}f_n(j\ze_\ell)
\ee
 to satisfy the given error tolerance. A good approximation to  $\La:=M_0\ze$ is 
\bbe\label{eqLa_z}
\La=\frac{1}{n-a_{\tV}}\ln\frac{C_{\tV}}{\eps}-\ln\frac{b}{2},
\ee
where $C_{\tV}$ and $a_{\tV}$ are from \eq{bound_tV_disc}. 
If $V_n$ are real, then $\overline{h(z)}=h(\bar z)$, and, therefore, we can replace \eq{Vn_inf_sinh_trunc} with
\bbe\label{Vn_inf_sinh_trunc_sym}
V_n\approx 2\ze_\ell \Re \sum_{j=0}^{M_0}f_n(j\ze_\ell)(1-\de_{j0}/2).
\ee
 The complexity of the numerical scheme is of the order of $(n+1-a_{\tV})^{-1}\ln(H(f_n,d)/\eps)\ln(1/\eps)$.   If double precision arithmetic is used, then the deformation must be chosen so that the $f(j\ze)$ are not very large. 
 Furthermore, the image of the strip $S_{(-d,d)}$ under the map $\chi_{L; \sg_\ell,b_\ell,\om_\ell}$
  has non-empty intersection with the unit disc, hence, if the parameters of the deformation are fixed, and $n$ increases, then
 $H(f_n,d)$ increases as $B^{n+1}$, where $B>1$ depends on the chosen deformation.
  Therefore, the problem of an accurate bound for the Hardy norm and choice of $\ze_\ell$ becomes non-trivial. 
  This difficulty can be alleviated if $\ga>\pi/4$, better, $\ga>\pi/2$ (we will see that in applications to pricing
  options with discrete monitoring, $\ga>\pi/2$) choosing $n$-dependent parameters of the deformation.
  \vskip0.1cm
  \noindent {\sc Case I.} 
 $\ga\in (\pi/4, \pi/2]$ or $\ga>\pi/2$ but $\ga-\pi/2$ is very small. We set $\om_\ell=3\pi/8-\ga/2$, and take $d_\ell\in (0,(\ga-\pi/4)/2)$, e.g., $d_\ell=0.95 \cdot(\ga-\pi/4)/2$.
 Next,
 \begin{enumerate}[(i)]
 \item
  if $A>1$ and $A-1$ is not very small, we  find $b_\ell$ and $\sg_\ell$
  solving the system $1=\sg_\ell-b_\ell\sin(\om_\ell+d_\ell), A=\sg_\ell-b_\ell\sin(\om_\ell-d_\ell)$.
  A fairly safe upper bound for $H(f_n,d)$ is $H_{\mathrm{appr.}}(f_n,d)=C_{\tV}\max\{1,B\}$, where
  $B$ is the supremum of $y$ s.t. $\chi_{L; \sg_\ell,b_\ell,\om_\ell}(i(\om_\ell+d_\ell)+y)\in \cD(0,1)$;
  \item
  if  $A$ is close to 1 (this is the case when the time interval between the monitoring dates
  is small), we set $R=1-5/n$, and find $b_\ell$ and $\sg_\ell$
  solving the system $R=\sg_\ell-b_\ell\sin(\om_\ell+d_\ell), 1=\sg_\ell-b_\ell\sin(\om_\ell-d_\ell)$.
  A fairly safe upper bound for $H(f_n,d)$ is $H_{\mathrm{appr.}}(f_n,d)=C_{\tV}R^{-n-1}B$, where
  $B$ is the supremum of $y$ s.t. $\chi_{L; \sg_\ell,b_\ell,\om_\ell}(i(\om_\ell+d_\ell)+y)\in \cD(0,1)$. If $\ga$ is close
  to $\pi/4$, it is necessary to replace $R^{-n-1}$ with $\sup_{0\le y\le B}|\chi_{L; \sg_\ell,b_\ell,\om_\ell}(i(\om_\ell+d_\ell)+y)|^{-n-1}$.
 \end{enumerate}
 \vskip0.1cm
  \noindent {\sc Case II.} 
 $\ga\in (\pi/2, \pi)$, and $\ga-\pi/2$ is not very small. We choose $\om_\ell=(\pi/2-\ga)/2$, and  $d_\ell\in (0,|\om_\ell|)$, e.g.,
 $d_\ell=0.95|\om_\ell|$.  Next,
 \begin{enumerate}[(i)]
 \item
 if $A>1$ and $A-1$ is not very small, we find $b_\ell$ and $\sg_\ell$
  solving the system $1=\sg_\ell-b_\ell\sin(\om_\ell-d_\ell), A=\sg_\ell-b_\ell\sin(\om_\ell+d_\ell)$.
  A fairly safe upper bound for $H(f_n,d)$ is $H_{\mathrm{appr.}}(f_n,d)=C_{\tV}\max\{1,B\}$, where
  $B$ is the supremum of $y$ s.t. $\chi_{L; \sg_\ell,b_\ell,\om_\ell}(i(\om_\ell-d_\ell)+y)\in \cD(0,1)$;
    \item
  if  $A$ is close to 1, we set $R=1-5/n$, and find $b_\ell$ and $\sg_\ell$
  solving the system $R=\sg_\ell-b_\ell\sin(\om_\ell-d_\ell), 1=\sg_\ell-b_\ell\sin(\om_\ell+d_\ell)$.
  A fairly safe upper bound for $H(f_n,d)$ is $H_{\mathrm{appr.}}(f_n,d)=C_{\tV}R^{-n-1}B$, where
  $B$ is the supremum of $y$ s.t. $\chi_{L; \sg_\ell,b_\ell,\om_\ell}(i(\om_\ell-d_\ell)+y)\in \cD(0,1)$. 
\end{enumerate}
 \vskip0.1cm
  \noindent {\sc Case III.} 
If $A<1$ and $1-A$ is not small, we suggest to make the change of variables $z=Az'$,  follow
Steps I and II, and choose the step $\ze_\ell$ and the number of terms $M_0$ using the error tolerance 
$\eps A^{-n-1+a_{\tV}}$.

See Fig. \ref{fig:Graph1} for  illustrations of Cases I(i) and II(ii).


\section{Expectations of functions of  random  walk and its extremum}\label{exp_Levy_extremum}
\subsection{The Wiener-Hopf factorization}\label{ss:WHF}
 Let $\cEq$ be the EPV-operator under $X$ defined by  $u(x)=\bE[u(X_{T_q})]$; the EPV operators $\cE^\pm_q$ are defined in the introduction.
 We realize the EPV operators $\cEq$ and $\cE^\pm_q$ as pseudo-differential operators  (PDO)\footnote{Recall that a PDO $A=a(D)$ with symbol $a$
 acts on a sufficiently regular functions as follows: $Au(x)=\cF^{-1}_{\xi\to x}a(\xi)\cF_{x\to \xi}u(x)$, where $\cF$ and $\cF^{-1}$ are the Fourier transform and its inverse.} with the symbols $(1-q)/(1-q\Phi(\xi))$ and $\phi^\pm_q(\xi)$, where 
  $\phipq(\xi)=\bE[e^{i\xi \barX_{T_q}}]$ and $\phimq(\xi)=\bE[e^{i\xi \uX_{T_q}}]$ are the Wiener-Hopf factors. 
  We use the following key result valid for random walks on $\bR$ and L\'evy processes $X$ on $\bR$ \cite{greenwood-pitmanRW,greenwood-pitman}; in the latter
case, $T_q$ is an exponentially distributed random variable on mean $q$, independent of $X$. See  \cite{Borovkov1,RW,sato} for the references to the literature on the Wiener-Hopf factorization and various fluctuation identities.
\begin{lem}\label{l:deep} Let $X$ and $T_q$ be as above.  Then
\begin{enumerate}[(a)]
\item the random variables $\barX_{T_q}$ and
$X_{T_q}-\barX_{T_q}$ are independent; and

\item the random variables $\uX_{T_q}$ and
$X_{T_q}-\barX_{T_q}$ are identical in law.
\end{enumerate}
\end{lem}
(By symmetry, the statements (a), (b) are valid with $\barX$ and $\uX$ interchanged).
  The two basic forms of
  the Wiener-Hopf factorization (both immediate from
 Lemma \ref{l:deep}) are   
 \bbe\label{eq:operWHF}
 \cEq=\cEpq\cEmq=\cEmq\cEpq,
 \ee
 and
 \bbe\label{eq:whf_random}
\frac{1-q}{1-q\Phi(\xi)}=\phipq(\xi)\phimq(\xi).
\ee
Explicit analytic formulas for the Wiener-Hopf factors are easy to derive if at least one tail of the pdf of $Y$ decays exponentially,
equivalently,
$\Phi$ admits analytic continuation to a strip $S_{[\mum,\mup]}$, where $\mum\le 0\le \mup$, and $\mum<\mup$. 
The formulas for and the properties of the Wiener-Hopf factors are well-known, see, e.g., \cite{Borovkov1,NG-MBS,IDUU};  we include a short proof in Sect. \ref{ss:proof_prop_WHF}.
\begin{prop}\label{prop_WHF} Let $\Phi$ admit analytic continuation to a strip $S_{[\lm,\lp]}$, where $\lm\le 0\le \lp$, and $\lm<\lp$.
Then, for any $q\in (0,1)$, 
\begin{enumerate}[(a)]
\item
there exist $\mum\ge \lm$ and $\mup\le \lp$ s.t. $\mum<\mup$, and $c>0$ such that
\bbe\label{eq:boundPhiq}
\Re (1-q\Phi(\xi))\ge c, \quad \xi\in S_{[\mum,\mup]}.
\ee
\item
Furthermore, for any $\xi$ in the half-plane $\{\Im\xi>\mum\}$ and any $\omm\in [\mum, \Im\xi)$,
\beqa\label{phip1}
 \phipq(\xi)&=&\exp\left[-\frac{1}{2\pi i}\int_{\Im\xi =\omm}\frac{\xi\ln((1-q)/(1-q\Phi(\eta)))}{\eta(\xi-\eta)}d\eta\right],
 \eqa
 and 
 for any $\xi$ in the half-plane $\{\Im\xi<\mup\}$ and any $\omp\in (\Im\xi,\mup]$,
\beqa\label{phim1}
 \phimq(\xi)&=&\exp\left[\frac{1}{2\pi i}\int_{\Im\xi =\omp}\frac{\xi\ln((1-q)/(1-q\Phi(\eta)))}{\eta(\xi-\eta)}d\eta\right];
 \eqa
 \item
 Let let there exist $\de>0$ such that $\Phi(\xi)=O(|\xi|^{-\de})$ as $(S_{[\mum,\mup]}\ni)\xi\to \infty$. Then  $\phi^\pm_q(\xi)=c^\pm_q+\phi^{\pm,\pm}_q(\xi)$, where $\phi^{\pm,\pm}_q(\xi)=O(|\xi|^{-\de+\eps})$ as $(S_{[\mum+\eps,\mup-\eps]}\ni)\xi\to \infty$, for any $\eps>0$, and $c^\pm_q$ are given by 
 \bbe\label{cpmq}
 c^\pm_q=\exp\left[\pm\frac{1}{2\pi i}\int_{\Im\eta=\om_\mp}\frac{\ln(1-q\Phi(\eta))}{\eta}\right],
 \ee
 where $\omm\in (\mum,0)$ and $\omp\in (0,\mup)$. If $\mum=0$, then $c^+_q=(1-q)/c^-_q$, and if $\mup=0$, then 
 $c^-_q=(1-q)/c^+_q$.
 \end{enumerate}
\end{prop} 
\begin{example}\label{ex:rate_Phi}{\rm 
Let $\Phi(\xi)=e^{-\barDe\psi(\xi)}$, where $\barDe>0$ is the time interval between the monitoring dates, 
and $\psi$ the characteristic exponent of a L\'evy process. Then, in the Variance Gamma model, $\Phi(\xi)=O(|\xi|^{-\barDe\de})$, where $\de>0$ depends on the parameters of the process, and in all other popular models, 
$\Phi(\xi)=O(e^{-\barDe c_\infty|\xi|^\nu})$, where $c_\infty>0$ and $\nu\in (0,2]$ (see \cite{EfficientAmenable}).
}
\end{example}
The integrands on the RHS' of the formulas for the Wiener-Hopf factors above decay slowly at infinity, hence, very long grids are necessary
to calculate the Wiener-Hopf factors. If $\Phi$ admits analytic continuation to the union of a strip and cone containing or adjacent to the real line, then the Wiener-Hopf factors can be calculated with almost machine precision using appropriate conformal deformations of the
lines of integration on the RHS' of \eq{phip1}-\eq{phim1}. See Sect.  \ref{sss:WHFI}.

\subsection{Main theorems}\label{ss:main}
Let $X$, $q$ and $T_q$ be as in the introduction. Let $f$ be measurable and uniformly bounded on $U_+:=\{(x_1, x_2)\ |\ x_2\ge 0, x_1\le x_2\}$. 
Consider $
V(f;n;x_1,x_2)=\bE[f(x_1+X_n, \max\{x_2, x_1+\barX_n\})]. 
$
We write the (modified) $Z$-transform \eq{Vnze} in the form 
\bbe\label{tVnze}
(1-q)\tV(q)=\bE[f(x_1+X_{T_q}, \max\{x_2, x_1+\barX_{T_q}\})].
\ee
Notationally, the Wiener-Hopf factorization technique for random walks is identical to the Wiener-Hopf
factorization technique for L\'evy processes. See, e.g., \cite{NG-MBS,IDUU}.
The following theorem is a counterpart of \cite[Thm. 3.1]{EfficientLevyExtremum} for L\'evy processes;
$I$ denotes the identity operator, $f_+$ is the extension of $f$ to $\bR^2$ by zero, and $\De$ is the diagonal map: $\De(x)=(x,x)$.
 
 \begin{thm}\label{thm:X_barX_exp}
 Let $X$ be a L\'evy process on $\bR$, $q>0$, and let $f:U_+\to \bR$ be a measurable and uniformly bounded  function
s.t.   $((\cEmq\otimes I)f)\circ \De:\bR\to\bR$ is measurable.
 Then
 \begin{enumerate}[(i)]
 \item
 for any $x_1\le x_2$,
 \beqa\label{tVq0}
 (1-q)\tV(f;q;x_1,x_2)&=&((\cEq\otimes I)f_+)(x_1,x_2)+(\cEpq w(f;q,\cdot, x_2))(x_1),
 \eqa
 where 
 \bbe\label{eq:wqVtq}
 w(f;q,y,x_2)=\bfo_{[x_2,+\infty)}(y)(((\cEmq\otimes I)f_+)(y,y)-((\cEmq\otimes I)f_+)(y,x_2));
 \ee
\item as a function of $q$, $\tV(f,q;x_1,x_2)$ admits analytic continuation to the open unit disc.
 \end{enumerate}
  \end{thm}
  \begin{proof}
We use Lemma \ref{l:deep}.
 By definition, part (a) amounts to the statement that the
probability distribution of the $\bR^2$-valued random variable
$(\barX_{T_q}, X_{T_q}-\barX_{T_q})$ is equal to the product (in the sense
of ``product measure'') of the distribution of $\barX_{T_q}$ and the
distribution of $X_{T_q}-\barX_{T_q}$. Hence, we can 
apply Fubini's theorem.
 For $x_1\le  x_2$, we have
\beqast
&&\bE[f_+(x_1+X_{T_q}, \max\{x_2, x_1+\barX_{T_q}\})]\\
&=&\bE[f_+(x_1+X_{T_q}-\barX_{T_q}+\barX_{T_q}, \max\{x_2, x_1+\barX_{T_q}\})]
\\ &=&
\bE[((\cEmq\otimes I)f_+)(x_1+\barX_{T_q}, \max\{x_2, x_1+\barX_{T_q}\})]
\\
&=&\bE[((\cEmq\otimes I)f_+)(x_1+\barX_{T_q}, x_2)]
\\
&&+
\bE[\bfo_{x_1+\bar X_{T_q}\ge  x_2}(((\cEmq\otimes I)f_+)(x_1+\barX_{T_q}, x_1+\barX_{T_q})-((\cEmq\otimes I)f_+)(x_1+\barX_{T_q}, x_2))].
\eqast
Using \eq{eq:operWHF}, we write the first term on the rightmost side as $((\cEq\otimes I)f_+)(x_1,x_2)$, and
finish the proof of (i).
 As operators acting in the space of bounded measurable functions, $\cE^\pm_q$ 
admit analytic continuation w.r.t. $q$ to the open unit disc,  which proves (ii).

  \end{proof}
  
 \begin{rem}{\rm  The inverse $Z$-transform of $(1-q)^{-1}(\cEq\otimes I)f_+(x_1,x_2)$ equals $\bE[f(x_1+X_T,x_2)]$,
and, therefore, can be easily calculated using the Fourier transform technique. Essentially, we have the price of the
European option of maturity $T$, the riskless rate being 0, depending on $x_2$ as a parameter. Thus, the new element is the calculation of the second term on the RHS of \eq{tVq0}. We calculate both terms in the same manner in order to facilitate the explanation of various blocks of our method.}
 \end{rem}

In exponential L\'evy models which are typically used in quantitative finance,
 payoff functions may increase exponentially, and options with discrete monitoring are typical situations where
 random walks appear implicitly. Hence, we  consider the action of the EPV-operators in $L_\infty(\bR; w)$, $L_\infty$- spaces with the weights 
 $w(x)=e^{\ga x}$, $\ga\in [\mum,\mup]$, and $w(x)=\min\{e^{\mum x},e^{\mup x}\}$, where $\mum\le0\le \mup, \mum<\mup$; the norm
 is defined by $\|u\|_{L_\infty(\bR;w)} =\|wu\|_{L_\infty(\bR)}$. The following theorem is the straightforward reformulation 
 of Theorem 3.2 in \cite{EfficientLevyExtremum}, the condition $q+\psi(i\ga)>0$ for the L\'evy process being replaced with 
 $1-q\Phi(i\ga)>0$. 
 The proof is the same.
\begin{thm}\label{thm:X_barX_exp_2}
 Let   a L\'evy process $X$ on $\bR$, function $f:U_+\to \bR$ and $q\in (0,1)$ satisfy the following conditions \begin{enumerate}[(a)]
 \item
there exist $\mum\le 0\le \mup$ such that  $\forall$\ $\ga\in [\mum,\mup]$, $\bE[e^{-\ga Y}]<\infty$ and $1-q\Phi(i\ga)>0$;
 \item
 $f$ is a measurable function admitting the bound
 \bbe\label{simple_bound_f_X_barX1}
 |f(x_1,x_2)|\le C(x_2)e^{-\mup x_1}, 
 \ee
 where $C(x_2)$ is independent of $x_1\le x_2$;
 \item
the function $((\cEmq\otimes I)f)\circ \De$ is measurable and admits the bound
 \bbe\label{simple_bound_f_X_barX2}
|((\cEmq\otimes I)f)(x_1,x_1)| \le Ce^{-\mum x_1} , 
 \ee
 where $C$ is independent of $x_1\ge 0$.

 \end{enumerate}
 Then the statements (i)-(iii) of Theorem \ref{thm:X_barX_exp} hold.
 
 \end{thm}
 
 \begin{rem}\label{rem:Cepq_term}{\rm
 Evaluating the RHS of  \eq{tVq0}, we will apply the Fourier transform and its inverse. If $f_+(x_1,x_2)$ is a piece-wise smooth
function of the first argument so that the Fourier transform (w.r.t. the first argument) decays not slower than $|\xi|^{-1}$ at infinity,
but $f_+(\cdot,x_2)$ has points of discontinuity,  then the composition of the Fourier transform and its inverse cannot recover
$f_+(\cdot,x_2)$ at the points of discontinuity. For instance, in the example of the joint cpdf, $f_+(x_1,x_2)=\bfo_{(-\infty,a_1]}(x_1)\bfo_{(-\infty,a_2]}(x_2)$,
where $a_1\le a_2$, is dicontinuous at $x_1=a_1$ and $x_2=a_2$. 
Hence, we  represent $\cEq$ in the form
$\cEq=(1-q)I+(1-q)q\Phi(D)(1-q\Phi(D))^{-1}$, and calculate the first term on the RHS of \eq{tVq0} as follows:
\bbe\label{tVq01}
 ((\cEq\otimes I)f_+)(x_1,x_2)=(1-q)f_+(x_1,x_2)
 +\frac{1}{2\pi}\int_{\Im\xi_1=\om}\frac{e^{ix_1\xi_1}(1-q)q\Phi(\xi_1)}{1-q\Phi(\xi_1)}
 \hat{(f_+)}_1(\xi_1,x_2)d\xi_1,
\ee
where  $\widehat{(f_+)}_1(\xi_1,x_2)=\cF_{x_1\to \xi_1} f_+(x_1,x_2)$ is the Fourier transform of $f_+$ w.r.t. the first argument, 
and admissible $\om\in (\mum,\mup)$ depend on the rate of increase of $f(x_1,x_2)$ as $x_1\to -\infty$. In particular, if $f$ is uniformly bounded, then any $\om\in (0,\mup)$ is admissible. If $\widehat{(f_+)}_1(\xi_1,x_2)=O(|\xi_1|^{-1})$ and $\Phi(\xi_1)=O(|\xi_1|^{-\de})$ as $\xi\to \infty$ along the line of integration, where $\de>0$, then the integrand on the RHS of \eq{tVq01} is of class $L_1$, and the integral defines a function continuous in $x_1$.
 
 }
 \end{rem}
Let
$V(G; h; n; x)$ be the price of the barrier option with
the payoff $G(X_n)$ at maturity and no rebate if the barrier $h$ is crossed before or at time $n$; the rsikless rate is 0.
Applying Theorem   \ref{thm:X_barX_exp_2} and Remark \ref{rem:Cepq_term}, we obtain
\begin{thm}\label{thm:X_barX_exp_2_barr}
 Let  a random walk $X$ on $\bR$ and $q\in (0,1)$ satisfy condition (a) of Theorem \ref{thm:X_barX_exp_2}, and let
  $G$ be a measurable function admitting the bound
$ |G(x)|\le C(e^{-\mup x}+e^{-\mum x})$, where $C$ is independent of $x\in \bR$.
 Then, for $x<h$,
 \bbe\label{eq:price_barr}
 \tV(G;h;q,x)=G(x)+(q\Phi(D)(1-q\Phi(D))^{-1}G)(x)-(1-q)^{-1}(\cEpq\bfo_{[h,+\infty)}\cEmq G)(x).
 \ee
 \end{thm}
 \begin{rem}{\rm The advantage of the representation \eq{eq:price_barr} as compared to the equivalent formula
 \bbe\label{gen_barrier}
 \tV(G;h;q,x)=(1-q)^{-1}(\cEpq \bfo_{(-\infty,h)}\cEmq)G(x)
 \ee
 (see \cite{IDUU} for the references) is that if $\hG(\xi)=O(|\xi|)^{-1}$ and $\Phi(\xi)=O(|\xi|^{-\de})$ as $\xi\to \infty$ in a strip around or adjacent to the real axis, where $\de>0$, then all the terms on the RHS of \eq{eq:price_barr} bar the first one are H\"older continuous
 on $(-\infty, h)$, and
numerical results are more accurate.
  }\end{rem}

\subsection{Fourier transform realization, the case $q\in (0,1)$}\label{s:evalFTposq01}
In this Subsection, $q\in (0,1)$ is fixed. The RHS' of the formulas for the Wiener-Hopf factors and  formulas that we derive below admit analytic continuation w.r.t. $q$ so that the inverse $Z$-transform can be applied.  
We use $\cE^{\pm}_q=c^\pm_qI+\cE^{\pm,\pm}_q$, where $\cE^{\pm,\pm}_q=\phi^{\pm,\pm}_q(D)$, and the equality
\[
w(f;q,x_1, x_2)=\bfo_{[x_2,+\infty)}(x_1)(((\cEmq\otimes I)f_+)(x_1,x_1)-((\cEmq\otimes I)f_+)(x_1,x_2))=0,\quad x_1\le x_2.
\]
to write the second term on
the RHS of \eq{tVq0} as 
\bbe\label{tVq02}
(\cEpq w(f;q,\cdot, x_2))(x_1)=(\cE^{++}_q w(f;q,\cdot, x_2))(x_1),
\ee
and  \eq{eq:wqVtq} as
\bbe\label{eq:wqVtq2}
w(f;q, y, x_2)=c^-_q w_0(y,x_2)+w^-(f;q, y, x_2),
\ee
where $w_0(y,x_2)=\bfo_{[x_2,+\infty)}(y)(f_+(y,y)-f_+(y,x_2))$, and 
\bbe\label{eq:wqVtqm}
w^-(f;q, y, x_2)=\bfo_{[x_2,+\infty)}(y)(((\cE^{--}_q \otimes I)f_+)(y,y)-((\cE^{--}_q \otimes I)f_+)(y,x_2)).
\ee
Substituting \eq{eq:wqVtq2} into \eq{tVq02}, we obtain
\bbe\label{tVq03}
(\cEpq w(f;q,\cdot, x_2))(x_1)
=c^-_q((\cE^{++}_q\otimes I) w_0)(x_1,x_2)+ ((\cE^{++}_q\otimes I) w^-)(f;q,x_1, x_2).
\ee
In order to derive explicit integral representations for the terms on the RHS of \eq{tVq03}, we impose
 the following conditions, which can be relaxed: 
\begin{enumerate}[(a)]
\item
condition (a) of Theorem \ref{thm:X_barX_exp_2} is satisfied;
\item
 there exist $\mumpr, \muppr\in(\mum,\mup)$, $\mumpr<\muppr$ such that $f$ admits bounds 
 \beqa\label{simple_bound_f_X_barX1pr}
 |f(x_1,x_2)|&\le & C(x_2)e^{-\muppr x_1}, \ x_1\le x_2,
 \\\label{simple_bound_f_X_barX2pr}
|((\cEmq\otimes I)f_+)(x_1,x_1)|&\le& Ce^{-\mumpr x_1} , \ x_1\in\bR,
 \eqa
  where $C(x_2)$ and $C$ are independent of $x_1\le x_2$, and $x_1\in\bR$, respectively; 
  \item
 for any $x_2$, there exists $C(x_2)>0$ such that 
 \beqa\label{boundf1}
 |\widehat {(f_+)}_1(\xi_1,x_2)|&\le&  C(x_2)(1+|\xi_1|)^{-1}, \quad \xi_1\in S_{[\muppr,\mup]},\\\label{boundfw}
 |\widehat {(w_0)}_1(\eta,x_2)|&\le&  C(x_2)(1+|\eta|)^{-1}, \quad \eta\in S_{[\mum,\mumpr]};
 \eqa
 \item
 there exists $C>0$  such that for $\xi_1\in S_{[\muppr,\mup]}$ and $\xi_2\in S_{[\mum,\mumpr]}$, 
 \bbe\label{boundf2}
 |\widehat {(f_+)}(\xi_1,\xi_2)|\le C(1+|\xi_1|)^{-1}(1+|\xi_2|)^{-1};
 \ee
 \item
 there exists $\de>0$ such that $\Phi(\xi)=O(|\xi|^{-\de})$ as $(S_{[\mum,\mup]}\ni)\xi\to \infty$.
\end{enumerate}
\begin{thm}\label{thm:mainFT}
Let conditions (a)-(e) hold.   Then, for any $\om, \om_1, \om_2$ and $\omm$ satisfying
\bbe\label{ompmom1om2}
\om, \om_1\in (\muppr,\mup),\
   \om_2\in (\mum,\mumpr),\ \omm\in (\mum,\om_1+\om_2),
 \ee 
 and $x_1\le x_2$,
 \beqa\label{eq:tVqmain}
\tV(f;q;x_1,x_2)&=&f(x_1,x_2)
+\frac{1}{2\pi}\int_{\Im\xi_1=\om}\frac{e^{ix_1\xi_1}q\Phi(\xi_1)}{1-q\Phi(\xi_1)}
 \hat{(f_+)}_1(\xi_1,x_2)\\\nonumber
 &&+\frac{c^-_q}{2\pi(1-q)}\int_{\Im\eta=\omm} e^{ix_1\eta}\phi^{++}_q(\eta)\widehat {(w_0)}_1(\eta,x_2) d\eta
 \\\nonumber
&&+\frac{1}{2\pi(1-q)}\int_{\Im\eta=\omm} e^{i(x_1-x_2)\eta}\phi^{++}_q(\eta)\widehat{w^-_0}(f;q,\eta,x_2)d\eta,
\eqa
where $\widehat{w^-_0}(f;q,\eta,x_2)$ is given by
\beqa\label{eq:wqtVq4}
&&
\widehat{w^-_0}(f;q,\eta,x_2)
\\\nonumber
&=&\frac{1}{2\pi}\int_{\Im\xi_1=\om_1} d\xi_1 \, 
\frac{e^{ix_2\xi_1}}{i(\xi_1-\eta)}\phi^{--}_q(\xi_1)(\widehat{f_+})_1(\xi_1,x_2)\\\nonumber
&&+\frac{1}{(2\pi)^2}\int_{\Im\xi_1=\om_1} \int_{\Im\xi_2=\om_2}d\xi_1\,d\xi_2\, \frac{e^{ix_2(\xi_1+\xi_2)}}{i(\eta-\xi_1-\xi_2)}
\phi^{--}_q(\xi_1)(\widehat{f_+})(\xi_1,\xi_2).
\eqa
\end{thm}
\begin{proof} Essentially, we repeat the proof of Theorem 4.1 in \cite{EfficientLevyExtremum}, with small necessary changes.
We calculate the terms on the RHS of \eq{tVq0}. The first two terms on the RHS of \eq{eq:tVqmain} follow from   \eq{tVq01}.
Consider the third term. Since \eq{boundfw} holds and $\phi^{++}_q(\eta)=O(|\eta|^{-\de_1})$ as $\eta\to\infty$ in the strip $S_{[\mum,\mup]}$, where $\de_1>0$,
the integral
\bbe\label{eq:cEppw0}
(\cE^{++}_q w_0(\cdot,x_2))(x_1)=\frac{1}{2\pi}\int_{\Im\eta=\omm} e^{ix_1\eta}\phi^{++}_q(\eta)\widehat {(w_0)}_1(\eta,x_2) d\eta
\ee
is absolutely convergent. It remains to consider $(\cE^{++}_q w^-(f;q,\cdot, x_2))(x_1)$.
If $\Im\eta=\om_-$,
\beqast
\widehat{w^-}(f;q,\eta,x_2)&=&-\int_{x_2}^{+\infty} dy\, e^{-iy\eta}\frac{1}{2\pi}\int_{\Im\xi_1=\om_1} d\xi_1 \, e^{i\xi_1 y}
\phi^{--}_q(\xi_1)(\widehat{f_+})_1(\xi_1,x_2)\\
&&+\int_{x_2}^{+\infty} dy\,e^{-iy\eta}\frac{1}{(2\pi)^2}\int_{\Im\xi_1=\om} \int_{\Im\xi_2=\om_2}d\xi_1\,d\xi_2\, e^{i(\xi_1+\xi_2) y}
\phi^{--}_q(\xi_1)(\widehat{f_+})(\xi_1,\xi_2).
\eqast
We apply Fubini's theorem to the first integral.   The integral $\int_{x_2}^{+\infty} dy\,e^{i(-\eta+\xi_1)y}=\frac{e^{ix_2(\xi_1-\eta)}}{i(\eta-\xi_1)}$ converges absolutely since $-\omm+\om_1>0$, and the repeated integral converges absolutely 
because $\phi^{--}_q(\xi)$ is uniformly bounded on the line of integration and \eq{boundf1} holds. Similarly, since 
$-\omm+\om_1+\om_2>0$, the integral $\int_{x_2}^{+\infty} dy\,e^{i(-\eta+\xi_1+\xi_2)y}=e^{ix_2(\xi_1+\xi_2-\eta)}/(i(\eta-\xi_1-\xi_2))$ converges absolutely. Since \eq{boundf2} holds, $\phi^{--}_q(\xi)=O(|\xi_1|^{-\de_1})$ as $\xi_1\to \infty$ along the line of integration, where $\de_1>0$,   and
 \bbe\label{bound_L1_1}
 \int_\bR \int_\bR d\xi_1\,d\xi_2\, (1+|\xi_1+\xi_2|)^{-1}(1+|\xi_1|)^{-1-\de_1}(1+|\xi_2|)^{-1}<\infty,
 \ee
 the Fubini's theorem is applicable to the second integral as well. Thus,
 \bbe\label{eq:wqtVq3}
 \widehat{w^-}(f;q,\eta,x_2)=e^{-i\eta x_2}\widehat{w^-_0}(f;q,\eta,x_2),
 \ee
 where $\widehat{w^-_0}(f;q,\eta,x_2)$ is given by \eq{eq:wqtVq4},
and we obtain the triple integral
\bbe\label{eq:cEppwm}
(\cE^{++}_q w^-(\cdot,x_2)(x_1)=\frac{1}{2\pi}\int_{\Im\eta=\omm} e^{i(x_1-x_2)\eta}\phi^{++}_q(\eta)\widehat{w^-_0}(f;q,\eta,x_2)d\eta.
\ee
The integrand admits a bound via $Cg(\eta,\xi_1,\xi_2)$, where
\[
g(\eta,\xi_1,\xi_2)=(1+|\eta|)^{-\de_1}(1+|\eta-\xi_1-\xi_2|)^{-1}(1+|\xi_1|)^{-1-\de_1}(1+|\xi_2|)^{-1}
\]
is of class $L_1(\bR^3)$ (see \cite[Eq.(3.24)]{EfficientLevyExtremum}). Substituting \eq{tVq01}, \eq{tVq03}, \eq{eq:cEppw0} and \eq{eq:cEppwm} 
   into \eq{tVq0}, we obtain \eq{eq:tVqmain}.

\end{proof}
\begin{rem}\label{rem:directwm} {\rm In standard situations such as in the two examples that we consider below,
the function $y\mapsto h(y)=(\cE^{--}_q\otimes I)f_+(y,y)-(\cE^{--}_q\otimes I)f_+(y,x_2)$ is
a linear combination of exponential functions (with the coefficients depending on $x_2$). Then $\widehat{w^-}(q;\eta,x_2)$ can be calculated directly, the double integral on the RHS of \eq{eq:wqtVq4} can be reduced to 1D integrals,
and the condition \eq{boundf2} replaced with the condition on $h$ similar to   \eq{boundfw}.
Analogous  simplifications are possible in more involved cases when $h$ is a piece-wise exponential polynomial in $y$.
}
\end{rem}

\subsection{Two examples}\label{s:two_examples}
 \subsubsection{Example I. The joint cpdf of $X_n$ and $\barX_n$.   }\label{ss:jointpdf}
 For $a_1\le a_2$, and $x_1\le x_2$, set $f(x_1,x_2)=\bfo_{(-\infty,\min\{a_1,x_2\}]}(x_1)\bfo_{(-\infty,a_2]}(x_2)$ and consider 
\[
  V(f;n,x_1,x_2)=\bQ[x_1+X_n\le a_1, x_2+\barX_n\le a_2].
  \] 
 If $x_2>a_2$, then $V(f;n,x_1,x_2)=0$. Hence, we assume that
  $x_2\le a_2$.   
    \begin{thm}\label{thm:CPDF}  Let $q\in (0,1)$, $a_1\le a_2, x_1\le x_2\le a_2$, 
and let the following conditions hold:
 \begin{enumerate}[(i)]
  \item 
  there exist $\mum< 0< \mup$ such that  $\forall$\ $\ga\in [\mum,\mup]$, $\bE[e^{-\ga Y}]<\infty$, and
  $1-q\Phi(i\ga)>0$;
  \item 
 there exists $\de>0$ such that $\Phi(\xi)=O(|\xi|^{-\de})$ as $(S_{[\mum,\mup]}\ni)\xi\to \infty$.
  \end{enumerate}
  Then, for any $\mum<\omm<0<\om_1<\mup$,  and $\om\in (0,\mup)$,
 \beqa\label{tVq3}
&& \tV(f;q,x_1,x_2)\\\nonumber
 &=&\bfo_{(-\infty,a_1]}(x_1) +\frac{1}{2\pi}\int_{\Im\xi_1=\om}\frac{e^{i(x_1-a_1)\xi_1}q\Phi(\xi_1)}{-i\xi_1(1-q\Phi(\xi_1))}d\xi_1
 \\\nonumber
 &&+\frac{1}{(2\pi)^2(1-q)}\int_{\Im\eta=\omm}d\eta\,e^{i(x_1-a_2)\eta}\phi^{++}_q(\eta)
  \int_{\Im\xi_1=\om_1}d\xi_1\, \frac{e^{i\xi_1(a_2-a_1)}\phi^{--}_q(\xi_1)}{\xi_1(\xi_1-\eta)}.
\eqa
    \end{thm}
    \begin{proof} We repeat the proof of Theorem 3.8 in \cite{EfficientLevyExtremum}
    with small necessary modifications. We have $f_+(x_1,x_2)=\bfo_{(-\infty,a_1]}(x_1)\bfo_{(-\infty,a_2]}(x_2), $ therefore, for $x_2\le a_2$, 
    \beqast
    w_0(y,x_2)&=&\bfo_{[x_2,+\infty)}(y)\bfo_{(-\infty,a_1]}(y)(\bfo_{(-\infty,a_2]}(y)-\bfo_{(-\infty,a_2]}(x_2))\\
    &=&-\bfo_{[x_2,+\infty)}(y)\bfo_{(-\infty,a_1]}(y)\bfo_{(a_2,+\infty)}(y)=0,
    \eqast
    hence, the third term on the RHS of \eq{eq:tVqmain} is 0. Next, 
    \[
    \widehat{(f_+)}_1(\xi_1,x_2)=\bfo_{(-\infty,a_2]}(x_2)\int_{-\infty}^{a_1}e^{-ix_1\xi_1}d\xi_1=\bfo_{(-\infty,a_2]}(x_2)\frac{e^{-ia_1\xi_1}}{-i\xi_1}d\xi_1
    \]
    is well-defined in the upper half-plane, and satisfies the bound \eq{boundf1} in any strip $S_{[\muppr, \mup]}$, where $\muppr\in (0,\mup)$.  Thus, the first two terms on the RHS of   \eq{eq:tVqmain} 
    are the first terms on the RHS of \eq{tVq3}.
      It remains to evaluate the double integral on the RHS of \eq{eq:tVqmain}. As mentioned in Remark \ref{rem:directwm}, in the present case, it is simpler to evaluate $w^-$, and then $\widehat{w^-}$, directly:  for any $x_2\le a_2$, $\om_1\in (0,\mup)$ and any $\eta\in \{\Im\eta\in (\mum,\om_1)\}$,
    \beqa\nonumber
    w^-(q,y,x_2)&=&\bfo_{(x_2,+\infty)}(y)(\cE^{--}_q\bfo_{(-\infty,a_1]})(y)(\bfo_{(-\infty,a_2]}(y)-1)\\\nonumber
    &=&-\bfo_{[a_2,+\infty)}(y)(\cE^{--}_q\bfo_{(-\infty,a_1]})(y)\\\nonumber
      &=&-\bfo_{(a_2,+\infty)}(y)\frac{1}{2\pi}\int_{\Im \xi_1=\om_1}d\xi_1\, e^{i(y-a_1)\xi_1}\frac{\phi^{--}(\xi_1)}{-i\xi_1},
    \eqa
    \beqa\label{hwm}
    \widehat{w^-}(q,\eta,x_2)&=&-\int_{a_2}^{+\infty}e^{-iy\eta}\frac{1}{2\pi}\int_{\Im \xi_1=\om_1}d\xi_1\, e^{i(y-a_1)\xi_1}\frac{\phi^{--}(\xi_1)}{-i\xi_1}\\\nonumber
    &=& -\frac{e^{-ia_2\eta}}{2\pi}\int_{\Im \xi_1=\om_1}d\xi_1\, e^{i(a_2-a_1)\xi_1}\frac{\phi^{--}(\xi_1)}{i(\eta-\xi_1)(-i\xi_1)}.
    \eqa
    It is easy to see that both integrals are absolutely convergent.
    Substituting \eq{hwm} into the double integral on the RHS of \eq{eq:tVqmain}, we obtain \eq{tVq3}.
    \end{proof}
     \begin{rem}{\rm If $x_1<a_1$, then it advantageous to move the line of integration in the first integral on the RHS of
    \eq{tVq3} down, and, on crossing the simple pole, apply the residue theorem. The first two terms
    on the RHS become $1/(1-q)$ plus  the integral over the line $\Im\xi_1=\omm$.
    }
    \end{rem} 
    \begin{rem}\label{rem:cpdf_simpl}{\rm
   The first step of the proof of Theorem \ref{thm:CPDF} implies that we can replace $\phi^{--}_q$ in the double integral on the RHS
    of \eq{tVq3} with $\phimq$. From the computational point of view, if we make the conformal change of variables, both changes do not lead to a significant increase in sizes of arrays necessary
    for accurate calculations, especially if $a_2-a_1>0$. The advantage is that it becomes unnecessary to evaluate $c^-_q$.
    Recall that the same $c^-_q$ appears for all $\xi_1$ in the formula $\phi^{--}_q(\xi_1)=\phimq(\xi_1)-c^-_q$,
    hence, it  is necessary to evaluate $c^-_q$ with a higher precision that $\phimq(\xi_1)$. At the same time, the integrand in the formula for
    $c^-_q$ decays slower at infinity than the integrand in the formula for $\phimq(\xi_1)$.
    }
    \end{rem}
    \begin{rem}\label{rem:no-touch}{\rm 
    Denote by $I_2(q;x_1,x_2)$ the double integral on the RHS of \eq{tVq3} multiplied by $1-q$.
    It follows from \eq{tVq02} that we can replace $\phi^{++}_q$ in the double integral 
     with $\phipq$. If $a_1<a_2$ and the conformal deformations are used, then this replacement causes no serious computational problems. If $a_1=a_2$, then the replacement leads to  errors typical for the Fourier inversion at points
    of discontinuity. However, in this case, 
     the RHS of \eq{tVq3} can be simplified as follows. We replace $\phi^{\pm,\pm}_q$ with
    $\phi^\pm_q$, which is admissible, then push the line of integration in the inner integral down, cross two simple poles at
    $\xi_1=0$ and $\xi_1=\eta$, and apply the residue theorem.  
    The double integral becomes the following 1D integral:
    \[
    I_2(q;x_1,x_2)=\frac{1}{2\pi}\int_{\Im\eta=\omm}d\eta\, e^{i(x_1-a_2)\eta}\frac{\phipq(\eta)(1-\phimq(\eta))}{-i\eta}.
    \]
    We push the line of integration to $\{\Im\eta=\om_1\}$ and use  the identity 
 $\phipq(\eta)\phimq(\eta)=(1-q)/(1-q\Phi(\eta))$ to obtain the
    formula for the perpetual  no-touch option: 
    \beqa\label{no-touch}
    (1-q)\tV(f,q;x_1,x_2)&=&    \frac{1}{2\pi}\int_{\Im\xi_1=\om_1}d\xi_1\, \frac{e^{i(x_1-a_2)\xi_1}\phipq(\xi_1)}{-i\xi_1},
    \ x_1\le x_2\le a_2.
    \eqa 
    Of course, \eq{no-touch} can be obtained using the main theorem directly.
        }
    \end{rem}

  \begin{rem}\label{rem:two_form_cpdf}{\rm 
   One  can push the line of integration in the outer integral on the RHS of \eq{tVq3} up and
  obtain 
  \beqast\label{I2Hilb}
  I_2(q;x_1,x_2)&=&\frac{1}{4\pi}\int_{\Im\xi_1=\om_1}d\xi_1\, e^{i(x_1-a_1)\xi_1}\frac{\phi^{++}_q(\xi_1)\phi^{--}_q(\xi_1)}{-i\xi_1}\\\nonumber
  &&+\frac{1}{(2\pi)^2}\mathrm{v.p.}\int_{\Im\eta=\om_1}d\eta\,e^{i(x_1-a_2)\eta}\phi^{++}_q(\eta)
  \int_{\Im\xi_1=\om_1}d\xi_1\, \frac{e^{i\xi_1(a_2-a_1)}\phi^{--}_q(\xi_1)}{\xi_1(\xi_1-\eta)},
\eqast
where $\mathrm{v.p.}$ denotes the Cauchy principal value. After that, one  can apply the fast Hilbert transform. However, the integrand decays very slowly at infinity, therefore, accurate calculations are possible only if very long grids are used, hence, the CPU cost is very large even for a moderate error tolerance.    }\end{rem}
    
    \subsubsection{Example II. Option to exchange the supremum for a power of the underlying} Let $\be>1$. Consider the option to exchange the  supremum $\bar S_n=e^{\bar X_n}$ for the power $S_n^\be=e^{\be X_n}$.
 The payoff function  $f(x_1,x_2)=(e^{\be x_1}-e^{x_2})_+\bfo_{(-\infty,x_2]}(x_1)$ satisfies \eq{simple_bound_f_X_barX1pr}-\eq{simple_bound_f_X_barX2pr}
 with arbitrary $\muppr>0$, $\mumpr<-\be$.  The extension
$f_+$ is defined by the same analytical expression as $f$. 
\begin{prop}\label{prop:exch_beta}
 Let $\be>1$ and let 
 conditions of Theorem \ref{thm:mainFT} hold  with $\mum<-\be, \mup> 0$.
 Then, for  $x_1\le x_2$, and any $0<\om_1<\mup$, $\mum<\omm<-\be$,
 \bbe\label{qtVbe2}
 \tV(f;q,x_1,x_2)=(1-q)^{-1}(e^{\be x_1}-e^{x_2})_++I_2(q,x_1,x_2)+(1-q)^{-1}\sum_{j=3,4}I_j(q,x_1,x_2),
 \ee
 where
 $I_j(q,x_1,x_2)$, $j=2,3,4,$ are given by \eq{eq:I2}, \eq{eq:I3} and \eq{eq:I4} below.
\end{prop}
\begin{proof}
We apply Theorem \ref{thm:mainFT} with
 $\muppr\in (0,\mup)$, $\mumpr\in (\mum,-\be)$.
For $x_2> 0$ and $\xi\in\bC$, 
 \beqast\label{widetildef1be}
  (\widehat{f_+})_1(\xi_1, x_2)&=&\int_{x_2/\be}^{x_2}e^{-ix_1\xi_1}(e^{\be x_1}-e^{x_2})dx_1\\
  &=&\frac{e^{x_2(\be-i\xi_1)}-e^{x_2(\be-i\xi_1)/\be}}{\be-i\xi_1}-e^{x_2}\frac{e^{-ix_2\xi_1}-e^{-ix_2\xi_1/\be}}{-i\xi_1}\\
  &=&e^{-ix_2\xi_1}\left(\frac{e^{x_2\be}}{\be-i\xi_1}
  +\be \frac{e^{x_2(1+i\xi_1(1-1/\be))}}{(\be-i\xi_1)(-i\xi_1)}-\frac{e^{x_2}}{-i\xi_1}\right),
  \eqast
  hence, the second term on the RHS of \eq{eq:tVqmain} equals
  \bbe\label{eq:I2}
  I_2(q,x_1,x_2)=
  \frac{1}{2\pi}\int_{\Im\xi_1=\omm}d\xi_1\,\frac{e^{i(x_1-x_2)\xi_1}q\Phi(\xi_1)}{1-q\Phi(\xi_1)}
 \left(\frac{e^{x_2\be}}{\be-i\xi_1}
  +\be \frac{e^{x_2(1+i\xi_1(1-1/\be))}}{(\be-i\xi_1)(-i\xi_1)}-\frac{e^{x_2}}{-i\xi_1}\right).
  \ee
Then we calculate 
\beqast
w_0(y,x_2)&=&\bfo_{[x_2,+\infty)}(y)((e^{\be y}-e^{y})-(e^{\be y}-e^{x_2}))=\bfo_{[x_2,+\infty)}(y)(e^{x_2}-e^{y}),
\\
\widehat{w_0}(\eta,x_2)&=&\int_{x_2}^{+\infty} e^{-iy\eta}(e^{x_2}-e^{y})dy=
 \frac{e^{x_2-ix_2\eta}}{i\eta(1-i\eta)},
\eqast
and the third term on the RHS of \eq{eq:tVqmain}:
\bbe\label{eq:I3}
I_3(q,x_1,x_2)=c^-_q\frac{e^{x_2}}{2\pi}\int_{\Im\eta=\omm}d\eta\, e^{i(x_1-x_2)\eta}\frac{\phi^{++}_q(\eta)}{i\eta(1-i\eta)}.
\ee
Next, we calculate $\hat w^-(q,\eta,x_2)$:
\beqast
\hat w^-(q,\eta,x_2)&=&\int_{x_2}^{+\infty}e^{-iy\eta}\frac{1}{2\pi}\int_{\Im\xi_1=\om_1}d\xi_1\,e^{iy\xi_1}\phi^{--}_q(\xi_1)
\left[\frac{e^{(\be-i\xi_1)y}-e^{(\be-i\xi_1)x_2}}{\be-i\xi_1}\right.\\
&&\left.+\be\frac{e^{(1-i\xi_1/\be)y}-e^{(1-i\xi_1/\be)x_2}}{(\be-i\xi_1)(-i\xi_1)}-\frac{e^{(1-i\xi_1)y}-e^{(1-i\xi_1)x_2}}{-i\xi_1}\right]\\
&=&\frac{e^{-ix_2\eta}}{2\pi}\int_{\Im\xi_1=\om_1}\phi^{--}_q(\xi_1)\left[\frac{e^{(\be-i\xi_1)x_2}}{\be-i\xi_1}\left(\frac{1}{i(\eta-\xi_1)-(\be-i\xi_1)}-\frac{1}{i(\eta-\xi_1)}\right)\right.\\
&&\hskip2.5cm +\frac{\be e^{(1-i\xi_1/\be)x_2}}{(\be-i\xi_1)(-i\xi_1)}\left(\frac{1}{i(\eta-\xi_1)-(1-i\xi_1/\be)}-\frac{1}{i(\eta-\xi_1)}\right)
\\
&&\hskip2.5cm\left.-\frac{e^{(1-i\xi_1)x_2}}{-i\xi_1}\left(\frac{1}{i(\eta-\xi_1)-(1-i\xi_1)}-\frac{1}{i(\eta-\xi_1)}\right)\right]\\
&=&\frac{e^{-ix_2\eta}}{2\pi}\int_{\Im\xi_1=\om_1}d\xi_1\,\frac{\phi^{--}_q(\xi_1)}{i(\eta-\xi_1)}\left[\frac{e^{(\be-i\xi_1)x_2}}{i\eta-\be}
+\frac{\be e^{(1-i\xi_1/\be)x_2}(1-i\xi_1/\be)}{(\be-i\xi_1)(-i\xi_1)(i\eta-1-i\xi_1(1-1/\be))}\right.\\
&&\hskip4.5cm \left.-\frac{e^{(1-i\xi_1)x_2}(1-i\xi_1)}{(-i\xi_1)(i\eta-1)}\right],
\eqast
and, finally, the double integral on the RHS of \eq{eq:tVqmain}:
\beqa\label{eq:I4}
&&I_4(q,x_1,x_2)\\\nonumber&=&\frac{1}{(2\pi)^2}\int_{\Im\eta=\omm}d\eta., e^{i(x_1-x_2)\eta}\phi^{++}_q(\eta)
\int_{\Im\xi_1=\om_1}d\xi_1\,e^{-ix_2\xi_1}\frac{\phi^{--}_q(\xi_1)}{i(\eta-\xi_1)}\\\nonumber
&&\cdot\left[\frac{e^{\be x_2}}{i\eta-\be}
+\frac{\be e^{(1+i\xi_1(1-1/\be))x_2}(1-i\xi_1/\be)}{(\be-i\xi_1)(-i\xi_1)(i\eta-1-i\xi_1(1-1/\be))}
-\frac{e^{x_2}(1-i\xi_1)}{(-i\xi_1)(i\eta-1)}\right].
\eqa

\end{proof}
\section{Efficient Fourier transform realizations}\label{s:evalFT}
\subsection{Conformal deformations }\label{ss:conf_def_1}
The integrals on the RHS of \eq{tVq3}, and, especially, in the formulas for the Wiener-Hopf factors,
decay very slowly at infinity, therefore, very long grids are needed to satisfy even a moderate error tolerance.
The sizes of the grids drastically decrease if the conformal deformations of the lines of integration with the subsequent conformal changes of variables and application of the simplified trapezoid rule are used, as in \cite{paraLaplace,paired,Contrarian}, where options with continuous monitoring are considered. Below, we adjust the constructions from  \cite{paraLaplace,paired,Contrarian}  to random walks,
with an additional twist:  in the case of finite variation processes with non-zero drift, 
in some situations, it may be necessary to use not the sinh-acceleration but
another family of apparently inferior deformations considered in \cite{ConfAccelerationStable}.

For $\gam\le 0\le \gap$, $\gam<\gap$, set $\cC_{\gam,\gap}=\{\rho e^{i\varphi}\ |\ \rho>0, \varphi\in (\pi-\gap, \pi-\gam)\cup (\gam,\gap)\}$.
As it is shown in \cite{SINHregular,EfficientAmenable}, in wide classes of L\'evy models, the characteristic functions $\Phi_{\barDe}$ of $X_{\barDe}$, where $\barDe>0$ is the time interval between monitoring dates, are {\em sinh-regular}. This   means that there exist $C, c>0$, $\nu\in (0,2]$, $\mum\le 0\le \mup$ and $\gam\le 0 \le\gap$,
$\mum<\mup$, $\gam<\gap$,  independent of $\barDe$, such that
$\Phi_{\barDe}$ admits analytic continuation to $i(\mum,\mup)+(\cC_{\gam,\gap}\cup \{0\})$, and obeys the bound
\bbe\label{PhiSINHbound}
|\Phi_{\barDe}(\xi)|\le C \exp(-c\barDe|\xi|^\nu), \quad \xi\in i(\mum,\mup)+(\cC_{\gam,\gap}\cup \{0\}).
\ee
If $X$ is the Variance Gamma processes, the characteristic function decays slower at infinity:
\bbe\label{PhiSINHboundVG}
|\Phi_{\barDe}(\xi)|\le C (1+|\xi|)^{-c\barDe}, \quad \xi\in i(\mum,\mup)+(\cC_{\gam,\gap}\cup \{0\}).
\ee 
Typically, $c<1$ or even $<0.1$, hence, for the options with daily (or even weekly) monitoring, $\Phi_{\barDe}$ decays very slowly at infinity, for Variance Gamma processes and processes close to the Variance Gamma ($\nu>0$ close to 0), especially slowly.
This implies that even a moderate precision is impossible to achieve even at a large CPU cost, for options of long maturity especially.
The conformal deformation technique allows one to greatly increase the rate of the decay of the integrand at infinity.

If \eq{PhiSINHbound} or \eq{PhiSINHboundVG} hold, then it is possible to find appropriate conformal deformations of the contours of integration in all formulas. In the case of L\'evy processes of finite variation, with non-zero drift
$\mu$,
the characteristic function $\Phi_{\barDe}$ is of the form $\Phi_{\barDe}=e^{i\mu\barDe\xi}\Phi^0_{\barDe}$, where 
$\Phi^0_{\barDe}$ obeys the bound \eq{PhiSINHbound} or \eq{PhiSINHboundVG} in a cone $\cC_{\gam,\gap}$, where $\gam<0<\gap$,
with $\nu<1$.

\subsection{Evaluation of the Wiener-Hopf factors}\label{sss:WHFI} For $\om_1\in \bR, b>0$ and $\om\in (-\pi/2,\pi/2)$, introduce the map $y\mapsto \chi_{\om_1,b,\om}(y)=i\om_1+b\sinh(i\om+y)$.
For all $\xi$ above  the angle $i\mum+(e^{i(\pi-\gam)}\bR_+\cup e^{i\gam}\bR_+)$, we can find $\om_1^-\in \bR$, $b^->0$ and $\om^-\in (\gam, \pi/2)$ such that the contour $\cL_{\om^-_1,b^-,\om^-}:=\chi_{\om_1^-,b^-,\om^-}(\bR)$
is below $\xi$ but above the angle. Hence, we can deform the line of integration in \eq{phip1}
into $\cL_{\om^-_1,b^-,\om^-}$, make the change of variables $\eta=\eta^-(y):=\chi_{\om_1^-,b^-,\om^-}(y)$ and obtain
\bbe\label{phip1sinh}
\phipq(\xi)=\exp\left[-\frac{b^- }{2\pi i}\int_{\bR}\frac{\xi\ln[(1-q)/(1-q\Phi(\eta^-(y))]}{
\eta^-(y)(\xi-\eta^-(y))}\cosh(i\om^-+y)dy\right].
\ee
 Similarly, for any $\xi$ below the angle $i\mup+(e^{i(\pi-\gap)}\bR_+\cup e^{i\gap}\bR_+)$, we can find $\om_1^+\in \bR$, $b^+>0$ and $\om^+\in (-\pi/2, \gap)$ such that the contour $\cL_{\om^+_1,b^+,\om^+}:=\chi_{\om_1^+,b^+,\om^+}(\bR)$
is above $\xi$ but below the angle. Hence, we can deform the line of integration in \eq{phim1}
into $\cL_{\om^+_1,b^+,\om^+}$, make the change of variables $\eta=\eta^+(y):=\chi_{\om_1^+,b^+,\om^+}(y)$ and obtain
\bbe\label{phim1sinh}
\phimq(\xi)=\exp\left[\frac{b^+}{2\pi i}\int_{\bR}\frac{\xi\ln[(1-q)/(1-q\Phi(\eta^+(y))]}{
\eta^+(y)(\xi-\eta^+(y))}\cosh(i\om^++y)dy\right].
\ee
In order that the deformation be justified, it is necessary that, in the process of the deformation, the fractions under the
log-sign in \eq{phip1sinh} and  \eq{phim1sinh} do not equal 0 for all $q$ and $\eta$ of interest; in order to avoid complications stemming from the analytic continuation to an appropriate Riemann surface, it is advisable to make sure that the fraction does not assume values in $(-\infty,0]$ in the process of the deformation. See Fig. \ref{WHF_disc_SINH_cond_nu1.2}
for an illustration. 

\vskip0.1cm
\noindent
{\sc Choice of $\om^\pm.$}
If $\gam<0<\gap$, then it is possible 
to  choose $\om^-\in (\gam,0)$ and $\om^+\in (0,\gap)$. If $\gam=0$, then both $\om^\pm\in (0,\gap)$, and if $\gap=0$, then
both $\om^\pm\in (\gam,0)$.  When the double integral on the RHS of \eq{tVq3} is evaluated, we need
to calculate the Wiener-Hopf factors at the points on two contours $\cL^\pm:=\cL_{\om^\pm_1, b^\pm, \om^\pm}$. In order to increase
the width of of the strip of analyticity of each of the integrands on the RHS' of \eq{phip1sinh} and \eq{phim1sinh}, 
one should take $\om^-=\gam+(\gap-\gam)/3$, $\om^+=\gap-(\gap-\gam)/3$.

 In the case of L\'evy processes of finite variation, with non-zero drift
$\mu$,
the characteristic function $\Phi_{\barDe}$ is of the form $\Phi_{\barDe}=e^{i\mu\barDe\xi}\Phi^0_{\barDe}$, where 
$\Phi^0_{\barDe}$ obeys the bound \eq{PhiSINHbound} or \eq{PhiSINHboundVG} in a cone $\cC_{\gam,\gap}$, where $\gam<0<\gap$,
with $\nu<1$.
If $\mu>0$, $\Phi_{\barDe}$ obeys the bound \eq{PhiSINHbound} or \eq{PhiSINHboundVG} in the cone $\cC_{0,\gap}$, and if $\mu<0$, then in the cone $\cC_{\gam,0}$. 
 If $\mu>0$, it is advantageous to calculate $\phimq(\xi)$ using \eq{phim1sinh} with $\om^+>0$, and then, if $\phipq(\xi)$ is needed,  use the Wiener-Hopf factorization identity. 
If $\mu<0$, it is advantageous to calculate $\phipq(\xi)$ using \eq{phip1sinh} with $\om^-<0$, and then, if $\phimq(\xi)$ is needed,  use \eq{eq:whf_random}.

\subsection{Evaluation of the integrals on the RHS of \eq{tVq3}}\label{evalv1D} 

If $x_1-a_1\ge 0$, it is advantageous to deform the line of integration
upwards into a contour of the form $\cL_{\om^+_1,b^+,\om^+}$, where $\om^+>0$, and if $x_1-a_1\le 0$, then into a 
a contour of the form $\cL_{\om^-_1,b^-,\om^-}$, where $\om^-<0$. If $x_1-a_1=0$, then any $\om\in (\gam,\gap)$ is admissible, and
$\om=(\gam+\gap)/2$ is (approximately) optimal. However, if $\Phi_{\barDe}$ is of the form $\Phi_{\barDe}=e^{i\mu\barDe\xi}\Phi^0_{\barDe}$, where 
$\Phi^0_{\barDe}$ obeys the bound \eq{PhiSINHbound} or \eq{PhiSINHboundVG} in a cone $\cC_{\gam,\gap}$, where $\gam<0<\gap$,
with $\nu<1$
 and $\mu>0$,
then the deformation with $\om^-<0$ is impossible because, for $|q|=R<1$, 
$1-q e^{i\barDe\mu\xi}\Phi^0_{\barDe}(\xi)$ equals 0 for some $\xi$ in the process of deformation. In this case, as in \cite{ConfAccelerationStable},
we use a less efficient family of conformal maps of the form
\bbe\label{sub-pol}
\chi_{S;\om,m, a}(y)=(y+i\om)\ln^m(a^2+(y+i\om)^2),
\ee
where $\om\in \bR, a>|\om|,$ and $m\ge 1$ is an integer. As $y\to\pm \infty$,
\[
\chi_{S;\om,m, a}(y)=(2\ln y)^m(y+i\om(1+m/\ln |y|)+O(|y|^{-1}),
\]
therefore, if we take $\om<0$ and change the variable $\xi=\chi_{S;\om,a,a}(y)$, then the exponent 
$e^{i\barDe\mu\xi(y)}$ increases as $y\to \infty$ in a strip around $\bR$ slower than the factor $\Phi_{\barDe}^0(\xi(y))$ decays
at infinity, and the product decays faster than prior to the change of variables. If $x_1-a_1>0$, we use $\om>0$.

Consider the repeated integral. Since $x_2-a_2<0$, in the outer
integral, we deform the line of integration so that the wings of the deformed contour point downwards. If 
the bound \eq{PhiSINHbound} (or \eq{PhiSINHboundVG}) holds in a cone $\cC_{\gam,\gap}$ where $\gam<0$, we use
the map $\chi_{\om^-_1,b^-,\om^-}$ with $\om^-<0$. As in the case of 1D integral, it may be necessary to use the map
$\chi_{S;\om,m, a}$ with $\om<0$. Since $a_2-a_1\ge 0$, in the inner integral, we deform the line of integration so that the wings of the deformed contour point upwards. If 
the bound \eq{PhiSINHbound} (or \eq{PhiSINHboundVG}) holds in a cone $\cC_{\gam,\gap}$ where $\gap>0$, we use
the map $\chi_{\om^+_1,b^+,\om^+}$. As in the case of 1D integral, it may be necessary to use the map 
$\chi_{S;\om,m, a}$ with $\om>0$. Note that a less efficient family of deformations must be used at most once in the 1D-integral, and
at most once in the repeated integral, and, in all cases, the Wiener-Hopf factors can be calculated using the sinh-acceleration. 
 
If \eq{PhiSINHbound} or \eq{PhiSINHboundVG} hold, then it is possible to find appropriate conformal deformations of the contours of integration in all formulas. In the case of L\'evy processes of finite variation, with non-zero drift
$\mu$,
the characteristic function $\Phi_{\barDe}$ is of the form $\Phi_{\barDe}=e^{i\mu\barDe\xi}\Phi^0_{\barDe}$, where 
$\Phi^0_{\barDe}$ obeys the bound \eq{PhiSINHbound} or \eq{PhiSINHboundVG} in a cone $\cC_{\gam,\gap}$, where $\gam<0<\gap$,
with $\nu<1$, then the conformal deformation of the contour of integration in the $Z$-inversion formula is impossible, and only trapezoid rule can be applied.

\section{Algorithm and numerical examples}\label{s:numer}
We take $x_1=x_2=0$ and calculate the joint cpdf $F(T,a_1,a_2)= V(T,a_1,a_2; 0, 0)$ assuming that the cone of analyticity
contains the real line: $\gam<0<\gap$. This allows us to use two contours in $\xi_1$ and $\eta$ planes for all purposes, one in the lower half-plane, the other in the upper half-plane. 
If either $\gam=0$ or $\gap=0$, then, firstly, in \eq{tVq3}, one of the lines of integration can be deformed using the sinh-map, but the
other line
can deformed using a less efficient family of deformations only (see Sect. \ref{evalv1D}), and, secondly, for the calculation of the Wiener-Hopf factors, an additional ``sinh-deformed" contour is needed. Hence, the total number of the contours is three, not two, as in the algorithm below.

\begin{enumerate}[Step I.]
\item 
 Following the recommendation in Sect.\ref{s:eff_inverse_ze}, choose either the parameters for the trapezoid rule
$M_0$ and $M=2*M_0+1$ and construct the grid $\vec{q}=R*\exp((i*2*\pi/M)*(0:1:M_0))$ or
choose the sinh-deformation and grid for the simplified trapezoid rule: $\vec{y}=\ze_\ell*(0:1:M_0)$,
$\vec{q}=\sg_\ell+i*b_\ell*\sinh(i*\om_\ell+\vec{y})$.
Calculate  the derivative $\vec{der_\ell}=i*b_\ell*\cosh(i*\om_\ell+\vec{y})$.
Note that if double precision arithmetic is used, the choice of $R$, $\sg_\ell$ and $b_\ell$ must depend on $T$ but can be independent of $x_1,x_2, a_1,a_2$, at some loss in the efficiency of the algorithm. For large $n'$s, this leads to a significant increase of
the number of terms in the trapezoid rule. In the case of the sinh-acceleration, the effect is less pronounced but
leads to worse results for very large $n$, as in the numerical examples for $T=15$ below. 
\item
Choose  the sinh-deformations and grids for the simplified trapezoid rule on $\cL^\pm$: $\vec{y^\pm}=\ze^\pm*(-N^\pm:1:N^\pm)$,
$\vec{\xi^\pm}=i*\om_1^\pm+ b^\pm*\sinh(i*\om^\pm+\vec{y^\pm})$.  Calculate $\Phi^\pm=\Phi(\vec{\xi^\pm})$ and 
$\vec{der^\pm}=b^\pm*\cosh(i*\om^\pm+\vec{y^\pm}).
$
\item
Calculate the arrays $D^+=[1/(\xi^+_j-\xi^-_k)]$ and $D^-=[1/(\xi^-_k-\xi^+_j)]$ (the sizes are $(2*N^++1)\times (2*N^-+1)$
and $(2*N^-+1)\times (2*N^++1)$, respectively).
\item
{\sc The main block.}  For given $x_1,x_2, a_1,a_2$, in the cycle in $q\in \vec{q}$, evaluate
\begin{enumerate}[(1)]
\item
$\phipq$ at points of the grid $\cL^+$ and  $\phimq$ at points of the grid $\cL^-$: 
\[
\vec{\phi^\pm_q}=\exp\left[((\mp\ze^\pm*i*\ze^\mp/(2*\pi))*\vec{\xi^\pm}.*(\log((1-q)./(1-q\Phi^\mp))./\vec{\xi^\mp}.*\vec{der^\mp})*D^\pm)\right];
\]
\item
calculate $\phi^\pm_q$ at points of the grid $\cL^\mp$:
$
 \vec{\phi^\pm_{q,\mp}}=(1-q)./(1-q\Phi^\mp)./\vec{\phi^\mp_q};
 $
 \item
 evaluate the 2D integral on the RHS of \eq{tVq3}
 \beqast
 Int2(q)&=&((\ze^-*\ze^+/(2*\pi)^2)*(\exp(-i*a_2*\vec{\xi^-}).* \vec{\phi^+_{q,-}}.*\vec{der^-})*D^+)\\
 && *\mathrm{conj}((\exp((i*(a_2-a_1))*\vec{\xi^+}).* \vec{\phi^-_{q,+}}/\vec{\xi^+}.*\vec{der^+})').
 \eqast
 \item
 if $x_1-a_1>0$, use  arrays $\vec{\xi^+}, \vec{der^+}, \Phi^+$ to evaluate $Int1(q)$, the 1D integral on the RHS of \eq{tVq3}; 
   if $x_1\le a_1$, use  arrays $\vec{\xi^-}, \vec{der^-}, \Phi^-$ instead and add $1/(1-q)$;
 \item
 set $Int(q)=Int1(q)./(1-\vec{q})+Int2(q)$.
 \end{enumerate}
 \item
 Set  $Int(q_1)=Int(q_1)/2$.
 \item
 If the sinh-acceleration is used for the inverse $Z$-transform, calculate
 \[
V_n=(\ze_{\ell}/\pi)*real(\mathrm{sum}(\vec{q}.^{-n-1}.*Int(\vec{q}).*\vec{der_\ell})); \]
if the trapezoid rule is used, calculate
\[
V_n=(2/M)*real(\mathrm{sum}(\vec{q}.^{-n}.*Int(\vec{q})).\] 
\end{enumerate}

Numerical results are produced using Matlab R2017b on MacBook Pro, 2.8 GHz Intel Core i7, memory 16GB 2133 MHz. 
The CPU times reported below can be significantly improved because 
we use the same grids for the calculation of the Wiener-Hopf factors $\phi^\pm_q$ and evaluation of integrals on the RHS of  \eq{tVq3}.
However, $\phi^\pm_q$ need to be evaluated only once and used for all points $(a_1,a_2)$. But if $x_1-a_2$ and $a_2-a_1$ are not very small in absolute value, then much shorter grids can be used. See, e.g., examples in \cite{iFT,paraHeston,SINHregular,Contrarian}.
Therefore, if the arrays $(x_1-a_2,a_2-a_1)$ are large, then the CPU time can be decreased  using shorter arrays 
for calculation of the integrals on the RHS of  \eq{tVq3}. 
Furthermore, the main blocks of the program admit the trivial parallelization.

In the two examples that we consider, the characteristic function is $\Phi(\xi)=e^{-\barDe \psi(\xi)}$, where $\psi$ is the characteristic exponent 
 $\psi(\xi)=c\Gamma(-\nu)(\lp^\nu-(\lp+i\xi)^\nu+(-\lm)^\nu-(-\lm-i\xi)^\nu)$ of a KoBoL process\footnote{the class of processes constructed in \cite{genBS,KoBoL}; a subclass which was used in the numerical examples in
\cite{genBS,NG-MBS} was renamed CGMY model later.}, where
$\lp=1,\lm=-2$ and
(I) $\nu=0.2$, hence, the process is close to Variance Gamma process; (II) 
$\nu=1.2$, hence, the process is close to the Normal Inverse Gaussian process (NIG). In both cases, $c>0$ is chosen so that the second instantaneous moment $m_2=\psi^{\prime\prime}(0)=0.1$. The time step is $\barDe=1/252$ (daily monitoring).
For $X_0=\barX_0=0$, 
we calculate the joint cpdf $F(T,a_1,a_2):=V(T,a_1,a_2; 0,0)$ for $T=15$ in Case (II) and for $T=0.25,1,5,15$
in Case (I).  In both cases, $a_1$ is in the range $[-0.075,0.1]$ and $a_2$ in the range
$[0.025,0.175]$; the total number of points $(a_1,a_2)$, $a_1\le a_2$, is 44. We show the results for $T=0.25, 5$ and $T=15$ because
the errors, CPU times and sizes of arrays in the case $T=1$ can be approximated well by interpolation of the results for $T=0.25$ and $T=5$.

The numerical examples demonstrate the clear advantage of the sinh-acceleration applied to the inverse $Z$-transform vs the trapezoid rule; the advantage increases proportionally to the number of steps because the sinh-acceleration requires approximately the same number of terms of the simplified trapezoid rule whereas the number of terms in the trapezoid rule increases. Note, however,
that if high precision arithmetic is used then the trapezoid rule with much smaller number of terms can be used.

We also show the errors of the approximation of the continuous time model with the model with daily monitoring.
The probabilities in the continuous time model are calculated using the method in \cite{EfficientLevyExtremum}.
As expected, the approximation errors increase with the number of steps but remain fairly good even at $T=15$.

\section{Conclusion}\label{s:concl}

There exists a large body of literature devoted
to calculation of expectations $V(f;T;x_1,x_2)$ of functions of spot value $x_1$ of $X$ and its running maximum or minimum $x_2$ and related optimal stopping problems,
standard examples being  barrier and American options, and lookback options with barrier and/or American features. 
See, e.g., \cite{GSh,barrier-RLPE,NG-MBS,BLSIAM02,perp-Bermuda,amer-put-levy,AAP,AKP,KW1,kou,FusaiAbrahamsSguarra06,IDUU,single,MSdouble,BLdouble,beta,KudrLev11,paraLaplace,COS2,COS3,BIL,MarcoDiscBarr,HaislipKaishev14,FusaiGermanoMarazzina,kirkbyJCompFinance18,Kirkby-Nguen-Cui17,Linetsky08,feng-linetsky09,LiLinetsky2015,Contrarian,BSINH} and the bibilographies therein.
In many papers, in the infinite time horizon case, the Wiener-Hopf factorization technique in various forms is used, and the finite time
horizon problems are reduced to the infinite time horizon case using the Laplace transform or its discrete version. The present paper belongs to this strand of the literature. We consider random walks, equivalently, in the context of option pricing, barrier and lookback options with discrete monitoring. 

At the first step, 
as in \cite{FusaiAbrahamsSguarra06}, where barrier options with discrete monitoring in the Brownian motion model are priced, 
we use the $Z$-transform, which is the discrete time counterpart of the Laplace transform. The latter was used
in the continuous time case in a number of publications starting with \cite{barrier-RLPE,NG-MBS}. The first contribution of the present paper is the new numerical method
for the inverse $Z$-transform, which is more efficient than the trapezoid rule. In both continuous time and discrete time cases,
the application of the Laplace and $Z$-transforms reduces the problem to pricing the corresponding options in the infinite time horizon.
The second contribution of the present paper is a general formula for the expectation of a function of a random walk and its supremum process. The formula generalizes the formulas for the barrier options in the random walk and L\'evy models derived in
\cite{barrier-RLPE,NG-MBS,perp-Bermuda,IDUU,single}, and it is a counterpart of the general formula for the L\'evy processes
derived in \cite{EfficientLevyExtremum}. Both formulas use the expected present value operators (EPV-operators) technique,
which is the operator form of the Wiener-Hopf factorization. The last contribution of the paper is the set of efficient numerical realizations of
the general formulas, which we explain in detail in the case of the calculation of the joint probability distribution of the random walk
and its supremum. The numerical examples demonstrate that the method based on the sinh-acceleration for the
inverse $Z$-transform can achieve the accuracy of the order of E-14 and better using Matlab and Mac with moderate characteristics, in a second or fraction of a second, and the precision of the order of 
E-10 in 20-30 msec., for options of maturity in the range $T=0.25-15Y$. In all cases, the sizes of the arrays are  moderate. In particular, the
number of points used for the $Z$-transform inversion is of the order of 2-5 dozens or even fewer.
If the trapezoid rule is used, the size of arrays and CPU time increase with the maturity, and, for maturity $T=15$,
approximately 3,000 terms are needed. When the trapezoid rule is applied, the  CPU time is several times larger in all cases. We also compare the results in
the case of continuous monitoring using the methods developed in \cite{EfficientLevyExtremum} and demonstrate that in the case of daily monitoring, the relative differences are less than 1\% even for $T=15$ for a process close to the Variance Gamma, and less than 5\%
for a process close to NIG.

Other methods for pricing barrier and lookback options with discrete monitoring cannot achieve the precision E-10 even at a much larger CPU cost. COS method \cite{COS2,COS3} introduces an additional source of errors, and the errors accumulate very fast. As numerical examples in \cite{MarcoDiscBarr} show, the errors of COS can be of the oder of 10\% even for options of short maturity, and blow up
for maturity $T=1Y$. BPROJ method \cite{kirkbyJCompFinance18,Kirkby-Nguen-Cui17,BSINH} also introduces an error, which accumulates but not as fast as the error of COS. Furthermore, the error of the approximation of the transition density
in BPROJ method is in the norm of the Sobolev space $H^2(\bR)$, hence, very large for distributions close to the Variance Gamma -
and, for small monitoring intervals, in the case of the Variance Gamma model, the $H^2$-norm is $+\infty$ (see \cite{BSINH} for
the detailed analysis of COS, BROJ and filtering used in the literature to increase the speed of convergence - at the cost of additional errors). The Hilbert transform approach (see, e.g., \cite{feng-linetsky09,FusaiGermanoMarazzina}) requires long grids, and the grids have to be extremely long for small time intervals and processes of finite variation. In addition, it is very difficult to
accurately estimate the accumulation of errors. The method of \cite{MarcoDiscBarr}, where the calculations are in the state space,
allows one to derive sufficiently accurate error bounds and recommendations for the choice of the parameters of the numerical scheme.
However, the grids must increase with time to maturity, and, in the result, for options of maturity more than a year, even the precision of the order of E-05 requires much more CPU time than the method of the present paper.

\bibliography{LookbackDiscII.bbl}

\begin{thebibliography}{10}

\bibitem{AAP}
S.~Asmussen, F.~Avram, and M.R. Pistorius.
\newblock Russian and {A}merican put options under exponential phase-type
  {L}\'evy models.
\newblock {\em Stochastic Processes and their Applications}, 109(1):79--111,
  2004.

\bibitem{AKP}
F.~Avram, A.~Kyprianou, and M.R. Pistorius.
\newblock Exit problems for spectrally negative {L}\'evy processes and
  applications to ({C}anadized) {R}ussian options.
\newblock {\em Annals of Applied Probability}, 14(2):215--238, 2004.

\bibitem{Borovkov1}
A.A. Borovkov.
\newblock {\em Stochastic processes in queueing theory}.
\newblock Springer-Verlag, Berlin, 1976.

\bibitem{MSdouble}
M.~Boyarchenko and S.~Boyarchenko.
\newblock Double barrier options in regime-switching hyper-exponential
  jump-diffusion models.
\newblock {\em International Journal of Theoretical and Applied Finance},
  14(7):1005--1044, 2011.

\bibitem{BIL}
M.~Boyarchenko, M.~de~Innocentis, and S.~Levendorski\u{i}.
\newblock Prices of barrier and first-touch digital options in {L}\'evy-driven
  models, near barrier.
\newblock {\em International Journal of Theoretical and Applied Finance},
  14(7):1045--1090, 2011.
\newblock Available at SSRN: http://papers.ssrn.com/abstract=1514025.

\bibitem{single}
M.~Boyarchenko and S.~Levendorski\u{i}.
\newblock Prices and sensitivities of barrier and first-touch digital options
  in {L}\'evy-driven models.
\newblock {\em International Journal of Theoretical and Applied Finance},
  12(8):1125--1170, December 2009.

\bibitem{BLdouble}
M.~Boyarchenko and S.~Levendorski\u{i}.
\newblock Valuation of continuously monitored double barrier options and
  related securities.
\newblock {\em Mathematical Finance}, 22(3):419--444, July 2012.

\bibitem{genBS}
S.~Boyarchenko and S.~Levendorski\u{i}.
\newblock Generalizations of the {B}lack-{S}choles equation for truncated
  {L}\'evy processes.
\newblock Working Paper, University of Pennsylvania, April 1999.

\bibitem{KoBoL}
S.~Boyarchenko and S.~Levendorski\u{i}.
\newblock Option pricing for truncated {L}\'evy processes.
\newblock {\em International Journal of Theoretical and Applied Finance},
  3(3):549--552, July 2000.

\bibitem{barrier-RLPE}
S.~Boyarchenko and S.~Levendorski\u{i}.
\newblock Barrier options and touch-and-out options under regular {L}\'evy
  processes of exponential type.
\newblock {\em Annals of Applied Probability}, 12(4):1261--1298, 2002.

\bibitem{NG-MBS}
S.~Boyarchenko and S.~Levendorski\u{i}.
\newblock {\em Non-{G}aussian {M}erton-{B}lack-{S}choles {T}heory}, volume~9 of
  {\em Adv. Ser. Stat. Sci. Appl. Probab.}
\newblock World Scientific Publishing Co., River Edge, NJ, 2002.

\bibitem{BLSIAM02}
S.~Boyarchenko and S.~Levendorski\u{i}.
\newblock Perpetual {A}merican options under {L}\'evy processes.
\newblock {\em SIAM Journal on Control and Optimization}, 40(6):1663--1696,
  2002.

\bibitem{perp-Bermuda}
S.~Boyarchenko and S.~Levendorski\u{i}.
\newblock Pricing of perpetual {B}ermudan options.
\newblock {\em Quantitative Finance}, 2:422--432, 2002.

\bibitem{EPV}
S.~Boyarchenko and S.~Levendorski\u{i}.
\newblock American options: the {E}{P}{V} pricing model.
\newblock {\em Annals of Finance}, 1:267--292, 2005.

\bibitem{IDUU}
S.~Boyarchenko and S.~Levendorski\u{i}.
\newblock {\em Irreversible {D}ecisions {U}nder {U}ncertainty ({O}ptimal
  {S}topping {M}ade {E}asy)}.
\newblock Springer, Berlin, 2007.

\bibitem{paraLaplace}
S.~Boyarchenko and S.~Levendorski\u{i}.
\newblock Efficient {L}aplace inversion, {W}iener-{H}opf factorization and
  pricing lookbacks.
\newblock {\em International Journal of Theoretical and Applied Finance},
  16(3):1350011 (40 pages), 2013.
\newblock Available at SSRN: http://ssrn.com/abstract=1979227.

\bibitem{iFT}
S.~Boyarchenko and S.~Levendorski\u{i}.
\newblock Efficient variations of {F}ourier transform in applications to option
  pricing.
\newblock {\em Journal of Computational Finance}, 18(2):57--90, 2014.
\newblock Available at http://ssrn.com/abstract=1673034.

\bibitem{SINHregular}
S.~Boyarchenko and S.~Levendorski\u{i}.
\newblock Sinh-acceleration: Efficient evaluation of probability distributions,
  option pricing, and {M}onte-{C}arlo simulations.
\newblock {\em International Journal of Theoretical and Applied Finance},
  22(3), 2019.
\newblock DOI: 10.1142/S0219024919500110. Available at SSRN:
  https://ssrn.com/abstract=3129881 or http://dx.doi.org/10.2139/ssrn.3129881.

\bibitem{ConfAccelerationStable}
S.~Boyarchenko and S.~Levendorski\u{i}.
\newblock Conformal accelerations method and efficient evaluation of stable
  distributions.
\newblock {\em Acta Applicandae Mathematicae}, 169:711--765, 2020.
\newblock Available at SSRN: https://ssrn.com/abstract=3206696 or
  http://dx.doi.org/10.2139/ssrn.3206696.

\bibitem{Contrarian}
S.~Boyarchenko and S.~Levendorski\u{i}.
\newblock Static and semi-static hedging as contrarian or conformist bets.
\newblock {\em Mathematical Finance}, 3(30):921--960, 2020.
\newblock Available at SSRN: https://ssrn.com/abstract=3329694 or
  http://arxiv.org/abs/1902.02854.

\bibitem{EfficientLevyExtremum}
S.~Boyarchenko and S.~Levendorski\u{i}.
\newblock Efficient evaluation of expectations of functions of a l\'evy process
  and its extremum.
\newblock Working paper, June 2022.
\newblock Available at SSRN: https://ssrn.com/abstract=4140462 or
  http://arxiv.org/abs/4362928.

\bibitem{EfficientAmenable}
S.~Boyarchenko and S.~Levendorski\u{i}.
\newblock L\'evy models amenable to efficient calculations.
\newblock Working paper, June 2022.
\newblock Available at SSRN: https://ssrn.com/abstract=4116959 or
  http://arxiv.org/abs/4339862.

\bibitem{BSINH}
S.~Boyarchenko, S.~Levendorski\u{i}, J.L. Kirkby, and Z.~Cui.
\newblock {S}{I}{N}{H}-acceleration for {B}-spline projection with option
  pricing applications.
\newblock {\em International Journal of Theoretical and Applied Finance}, -(-),
  2022.
\newblock Available at SSRN: https://ssrn.com/abstract=3921840 or
  arXiv:2109.08738.

\bibitem{MarcoDiscBarr}
M.~de~Innocentis and S.~Levendorski\u{i}.
\newblock Pricing discrete barrier options and credit default swaps under
  {L}\'evy processes.
\newblock {\em Quantitative Finance}, 14(8):1337--1365, 2014.
\newblock Available at: DOI:10.1080/14697688.2013.826814.

\bibitem{COS3}
F.~Fang, H.~J{\"o}nsson, C.W. Oosterlee, and W.~Schoutens.
\newblock Fast valuation and calibration of credit default swaps under
  {L}{\'e}vy dynamics.
\newblock {\em Journal of Computational Finance}, 14(2):57--86, Winter 2010.

\bibitem{COS2}
F.~Fang and C.W. Oosterlee.
\newblock Pricing early-exercise and discrete barrier options by
  {F}ourier-cosine series expansions.
\newblock {\em Numerische Mathematik}, 114(1):27--62, 2009.

\bibitem{feng-linetsky09}
L.~Feng and V.~Linetsky.
\newblock Computing exponential moments of the discrete maximum of a {L}\'evy
  process and lookback options.
\newblock {\em Finance and Stochastics}, 13(4):501--529, 2009.

\bibitem{FusaiAbrahamsSguarra06}
G.~Fusai, I.D. Abrahams, and C.~Sguarra.
\newblock An exact analytical solution for discrete barrier options.
\newblock {\em Finance and Stochastics}, 10(1):1--26, 2006.

\bibitem{FusaiGermanoMarazzina}
G.~Fusai, G.~Germano, and D.~Marazzina.
\newblock Spitzer identity, {W}iener-{H}opf factorization and pricing of
  discretely monitored exotic options.
\newblock {\em European Journal of Operational Research}, 251(1):124--134,
  2016.
\newblock DOI:10.1016/j.ejor.2015.11.027.

\bibitem{greenwood-pitman}
P.~Greenwood and J.~Pitman.
\newblock Fluctuation identities for {L}\'evy processes and splitting at the
  maximum.
\newblock {\em Advances in Applied Probability}, 12(4):893--902, 1980.

\bibitem{greenwood-pitmanRW}
P.~Greenwood and J.~Pitman.
\newblock Fluctuation identities for random walk by path decomposition at the
  maximum.
\newblock {\em Advances in Applied Probability}, 12(2):291--293, 1980.

\bibitem{GSh}
X.~Guo and L.A. Shepp.
\newblock Some optimal stopping problems with nontrivial boundaries for pricing
  exotic options.
\newblock {\em J.Appl. Probability}, 38(3):647--658, 2001.

\bibitem{HaislipKaishev14}
G.G. Haislip and V.K. Kaishev.
\newblock Lookback option pricing using the {F}ourier transform {B}-spline
  method.
\newblock {\em Quantitative Finance}, 14(5):789--803, 2014.

\bibitem{kirkbyJCompFinance18}
J.L. Kirkby.
\newblock American and {E}xotic {O}ption {P}ricing with {J}ump {D}iffusions and
  other {L}\'evy processes.
\newblock {\em Journ. Comp. Fin.}, 22(3):13--47, 2018.

\bibitem{Kirkby-Nguen-Cui17}
J.L. Kirkby, D.~Nguen, and Z.~Cui.
\newblock A unified approach to bermudan and barrier options under stochastic
  volatility models with jumps.
\newblock {\em Journal of Economic Dynamics and Control}, 80(1):75--100, 2017.

\bibitem{kou}
S.G. Kou.
\newblock A jump-diffusion model for option pricing.
\newblock {\em Management Science}, 48(8):1086--1101, August 2002.

\bibitem{KW1}
S.G. Kou and H.~Wang.
\newblock First passage times of a jump diffusion process.
\newblock {\em Adv. Appl. Prob.}, 35(2):504--531, 2003.

\bibitem{KudrLev11}
O.~Kudryavtsev and S.Z. Levendorski\u{i}.
\newblock Efficient pricing options with barrier and lookback features under
  {L}\'evy processes.
\newblock Working paper, June 2011.
\newblock Available at SSRN: 1857943.

\bibitem{beta}
A.~Kuznetsov.
\newblock Wiener-{H}opf factorization and distribution of extrema for a family
  of {L}\'evy processes.
\newblock {\em Ann.Appl.Prob.}, 20(5):1801--1830, 2010.

\bibitem{amer-put-levy}
S.~Levendorski\u{i}.
\newblock Pricing of the {A}merican put under {L}\'evy processes.
\newblock {\em International Journal of Theoretical and Applied Finance},
  7(3):303--335, May 2004.

\bibitem{paraHeston}
S.~Levendorski\u{i}.
\newblock Efficient pricing and reliable calibration in the {H}eston model.
\newblock {\em International Journal of Theoretical and Applied Finance},
  15(7), 2012.
\newblock 125050 (44 pages).

\bibitem{paired}
S.~Levendorski\u{i}.
\newblock Method of paired contours and pricing barrier options and {C}{D}{S}
  of long maturities.
\newblock {\em International Journal of Theoretical and Applied Finance},
  17(5):1--58, 2014.
\newblock 1450033 (58 pages).

\bibitem{LiLinetsky2015}
L.~Li and V.~Linetsky.
\newblock Discretely monitored first passage problems and barrier options: an
  eigenfunction expansion approach.
\newblock {\em Finance and Stochastics}, 19(3):941--977, 2015.

\bibitem{Linetsky08}
V.~Linetsky.
\newblock Spectral methods in derivatives pricing.
\newblock In J.R. Birge and V.~Linetsky, editors, {\em Handbooks in OR \& MS,
  Vol. 15}, pages 223--300. Elsevier, New York, 2008.

\bibitem{RW}
L.C.G. Rogers and D.~Williams.
\newblock {\em Diffusions, {M}arkov {P}rocesses, and {M}artingales. {V}olume 1.
  {F}oundations}.
\newblock John Wiley \& Sons, Ltd., Chichester, 2nd edition, 1994.

\bibitem{sato}
K.~Sato.
\newblock {\em L\'evy processes and infinitely divisible distributions},
  volume~68 of {\em Cambridge Stud. Adv. Math.}
\newblock Cambridge University Press, Cambridge, 1999.

\bibitem{stenger-book}
F.~Stenger.
\newblock {\em Numerical {M}ethods based on {S}inc and {A}nalytic functions}.
\newblock Springer-Verlag, New York, 1993.

\end{thebibliography}

\appendix 

\section{Technicalities}\label{s:tech}
\subsection{Proof of Theorem \ref{disctraperror}}\label{ss:proof_disctraperror}
First, let $h(z) = z^m$ for some integer $m$.  Then $T_M(h) = 0$ if $M$
 does not divide $m$, and $T_M(h) = 1$ if $M$ divides $m$.  This is a standard exercise about sums of roots of unity.
  Under conditions of the theorem,  $h(z)$ has a Laurent series expansion $h(z) = \sum_{j\in \bZ} b_j z^j$ 
which converges uniformly on the unit circle.  Then $I(h) = b_0$ and $T_n(h)$ is the sum of $b_j$ for all $j$
 that are divisible by $M$.  Hence, $|T_M(h)-I(h)|$ is bounded by the sum of $|b_j|$, where $j$ ranges over all nonzero integers that are divisible by $M$.
We have
 \beqast
 b_j&=&\frac{1}{2\pi i}\int_{|z|=1}z^{-j-1}h(z)dz
 =\frac{1}{2\pi i}\int_{|z|=\rho}z^{-j-1}h(z)dz
 =\frac{1}{2\pi i}\int_{|z|=1/\rho}z^{-j-1}h(z)dz.
 \eqast
 Hence,
 \beqast
 \sum_{j>0}|b_{Mj}|&\le&\sum_{j>0} \rho^{-Mj-1}\int_{|z|=\rho}|h(z)|\frac{dz}{2\pi i}
 =\frac{\rho^{-M}}{1-\rho^{-M}}\frac{1}{2\pi i}\int_{|z|=\rho}|h(z)|\frac{dz}{z},
 \eqast
 and, similarly, 
 \[
 \sum_{j<0}|b_{Mj}|\le \frac{\rho^{-M+1}}{1-\rho^{-M}}\frac{1}{2\pi i}\int_{|z|=1/\rho}|h(z)|\frac{dz}{z}.
 \]
Adding the two inequalities finishes the proof.

\subsection{Proof of Proposition \ref{prop_WHF}}\label{ss:proof_prop_WHF}  (a) follows from the following three facts: $\Phi(0)=1$; $\Phi$ is continuous on $i[\lm,\lp]$;  and $|\Phi(\xi)|\le \Phi(-\Im\xi)$. 
(b) Take $\xi\in S_{(\mum,\mup)}$ and note that the integrands
are analytic in $S_{[\mum,\mup]}$, with the only simple pole at $\eta=\xi$ and decay as $|\eta|^{-2}$ as $(S_{[\mum,\mup]}\ni)\eta\to\infty$ (the apparent singularity at $\eta=0$ is removable). By the residue theorem,
\[
\ln\frac{1-q}{1-q\Phi(\eta)}=-\frac{1}{2\pi i}\int_{\Im\xi =\omm}\frac{\xi\ln\frac{1-q}{1-q\Phi(\eta)}}{\eta(\xi-\eta)}d\eta+
\frac{1}{2\pi i}\int_{\Im\xi =\omp}\frac{\xi\ln\frac{1-q}{1-q\Phi(\eta)}}{\eta(\xi-\eta)}d\eta,
\]
hence, \eq{eq:boundPhiq} holds for $\phi^{\pm,'}_q(\xi)$ given by the RHS' of \eq{phip1}-\eq{phim1} and all $\xi\in S_{(\mum,\mup)}$.
Since $\phi^{+,'}_q$  and $\phipq$ are analytic and uniformly bounded in the upper half-plane, and
$\phi^{-,'}_q$  and $\phimq$ are analytic and uniformly bounded in the lower half-plane, \eq{phip1}-\eq{phim1}
follow from the uniqueness of the Wiener-Hopf factorization.

 (c) The integrals on the RHS' of \eq{phip1} and \eq{phim1} do not change if we omit the factor $1-q$ under the log sign.
Using $\xi/(\eta(\xi-\eta))=1/\eta+1/(\xi-\eta)$, we conclude that it suffices to prove that, for any $\eps>0$ and $A>0$, there exists $C_{A,\eps}>0$ such that for any $\xi\in S_{[-A,A]}$,
 \bbe\label{boundWHFeps}
 \int_\bR \frac{(1+|\eta|)^{-\de}}{1+|\xi-\eta|}d\eta \le C_{A,\eps}(1+|\xi|)^{-\de+\eps}.
 \ee
 We consider the integrals over $I_1=\{\eta\ |\ |\eta|\le (1+|\xi|)/2\}$, $I_2=\{\eta\ |\  |\eta|\ge 2(1+|\xi|)\}$ 
 and $I_3=\{\eta\ |\ (1+|\xi|)/2\le |\eta|\le 2(1+|\xi|)\}$:
 \beqast
  I_1&\le & C(1+|\xi|)^{-1} \int_0^{(1+|\xi|)/2} (1+|\eta|)^{-\de}d\eta=C_1(1+|\xi|)^{-\de},\\
  I_2&\le & C(1+|\xi|)^{-1} \int_{2(1+|\xi|)}^{+\infty} |\eta|^{-1-\de}d\eta=C_1(1+|\xi|)^{-\de},\\
  I_3&\le & C(1+|\xi|)^{-\de} \int_{(1+|\xi|)/2}^{2(1+|\xi|)}|\eta-\xi|^{-1}d\eta\le C_1(1+|\xi|)^{-\de} \ln (1+|\xi|),
\eqast
where $C,C_1$ are independent of $\xi$.

 \section{Figures and tables}\label{ss:figures}
  \begin{figure}
\scalebox{0.75}
{\includegraphics{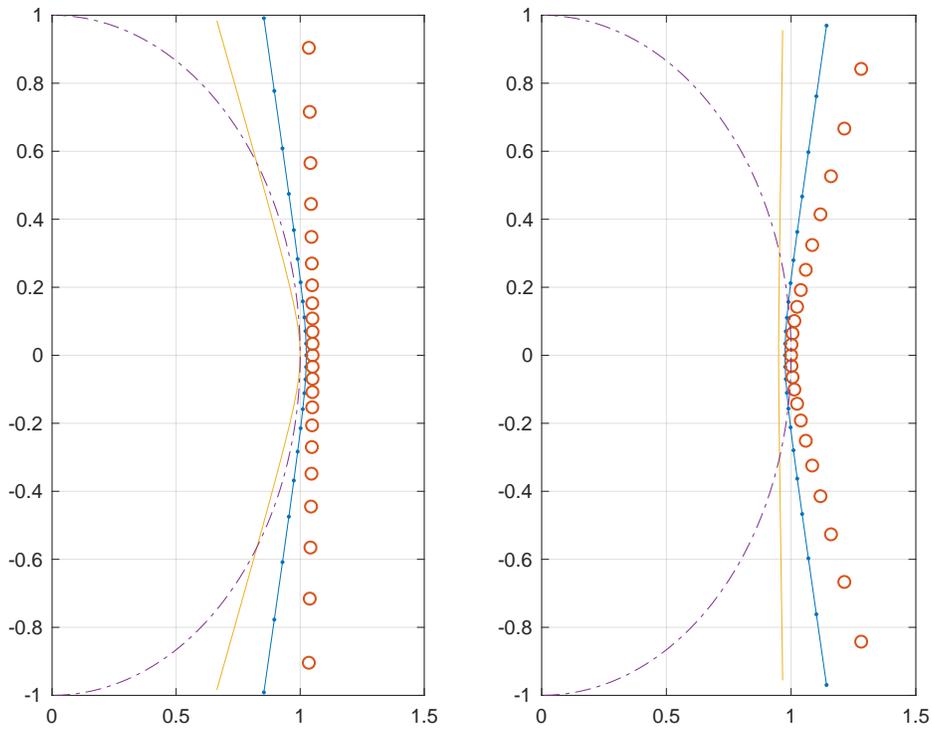}}
\caption{\small Cases I(i) (left panel) and II(ii) (right panel). Dots: the unit circle.
Dots-dashes, circles and solid lines: the curves $\chi_{L; \sg_\ell,b_\ell,\om_\ell}(i\om_\ell+\bR)$, $\chi_{L; \sg_\ell,b_\ell,\om_\ell}(i(\om_\ell+d_\ell)+\bR)$,
$\chi_{L; \sg_\ell,b_\ell,\om_\ell}(i(\om_\ell-d_\ell)+\bR)$.
}
\label{fig:Graph1}
\end{figure}

 \begin{figure}
\scalebox{0.75}
{\includegraphics{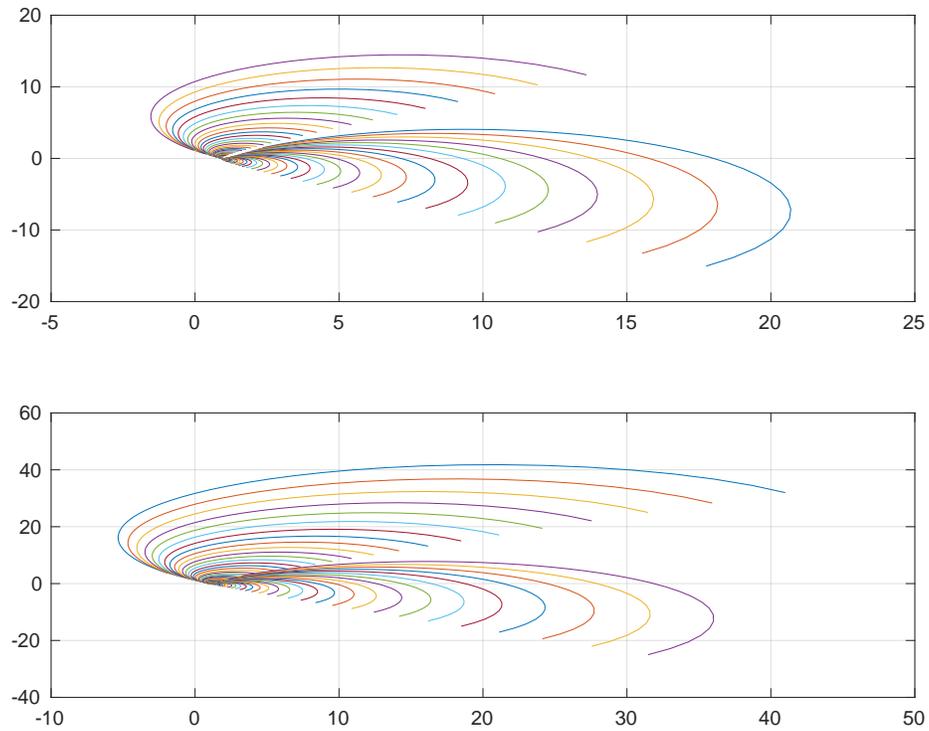}}
\caption{\small Plots of curves $\eta\mapsto (1-q)/(1-q\Phi(\eta))$, for $q$ in the SINH-$Z$- inversion and $\eta$ on the contours $\cL^\pm$  (upper and lower panels)
in the numerical example with $\nu=1.2$, and $T=15$.}
\label{WHF_disc_SINH_cond_nu1.2}
\end{figure}

 \begin{table}
\caption{\small Joint cpdf $F(T,a_1,a_2):=\bQ[X_T\le a_1, \barX_T\le a_2\ |\ X_0=\barX_0=0]$, and errors (rounded)
and CPU time (in msec) of two numerical schemes.  Discrete monitoring, the monitoring interval $\barDe=1/252$, $T=0.25Y$,
the number of time steps 63. KoBoL 
close to the Variance Gamma model, with an almost symmetric jump density, and no ``drift": $m_2=0.1$, $\nu=0.2, \lm=-2, \lp=1$. Errors are rounded, the CPU time is in milliseconds  (average over 1000 runs).
 }
 {\tiny
\begin{tabular}{c|c|cc|c|cc}
\hline\hline
$a_2/a_1$ & -0.075 & -0.05 & -0.025 & 0 & 0.025 \\\hline
0.025 & 0.052873910286366 & 0.0650091858382787 & 0.0879288341672031 & 0.506532201212114 & 0.923468308358369\\
0.05 & 0.0534088530783456 & 0.0656338924464693 & 0.0886847807216264 &0.507515090989102 & 0.925299214939269\\
0.075 & 0.0536456853005228 & 0.0659043877286091 & 0.0890004474115774 & 0.507896616129521 & 0.925793930891586\\
0.1 & 0.0537794257554031 & 0.0660548821001662 & 0.0891723010284717 & 0.508097111907463 & 0.926036138000489\\
0.175 & 0.0539628421387795 & 0.0662578446892915 & 0.0893989471374944 & 0.508353292242695 & 0.926330710592022
\\\hline
\end{tabular}

\begin{tabular}{c|ccccc|ccccc}
\hline\hline
&&& $A$ & & && & B & \\\hline
$a_2/a_1$ & -0.075 & -0.05 & -0.025 & 0 & 0.025 & -0.075 & -0.05 & -0.025 & 0 & 0.025
 \\\hline
0.025 &	4.03E-12	& 3.63E-12 &	2.61E-13 &	5.46E-12	& 1.88E-11 
& 4.41E-12 &	4.10E-12 &	3.4179E-13 &	-9.25E-13	& 1.38E-11\\

0.05	& 4.17E-12 &	3.81E-12	& 4.583E-13 &	5.80E-12 &	2.38E-12
& 4.57E-12 &	4.32E-12	& 6.20E-13 &	-5.57E-13 & 	-2.93E-12\\

0.075 & 4.09E-12 &	3.70E-12 &	3.26E-13 &	5.65E-12	& 3.14E-12
& 4.46E-12 &	4.18E-12	& 4.82E-13 &	-7.13E-13	& -3.14E-12
 \\
 
0.1 &	 3.89E-12	& 3.48E-12 &	5.87E-14 &	4.88E-12 &	1.78E-12
& 4.21E-12 &	3.91E-12	& 1.61E-13 &	-1.08E-12	& -3.56E-12
\\

0.175 & 4.03E-12 &	3.63E-12 &	2.31E-13	& 5.19E-12 &	1.10E-12
& 4.30E-12 &	4.0E-12 &	2.62E-13 &	-9.66E-13	& -3.44E-12
\end{tabular}
}
\begin{flushleft}{\tiny
Errors of the benchmark values: better than E-14, at some points,  E-15. CPU time per 1 point: 980, per 44 points: 6,672.\\
A:  Trapezoid rule, $M_0=99$, $N^\pm=124$. CPU time per 1 point: 30.9; per 44 points:
496.\\
B: SINH applied to the inverse $Z$-transform, with $M_0=16, N^\pm=124$. CPU time per 1 point 10.2, per 44 points: 73.5.}
\end{flushleft}

\label{table:cpdf0}
 \end{table}

  \begin{table}
\caption{\small Joint cpdf $F(T,a_1,a_2):=\bQ[X_T\le a_1, \barX_T\le a_2\ |\ X_0=\barX_0=0]$, in the continuous time model, and errors (rounded) of approximation by the discrete time model, with the time step $\barDe=1/252$. $T=0.25Y$. KoBoL 
close to the Variance Gamma model, with an almost symmetric jump density, and no ``drift": $m_2=0.1$, $\nu=0.2, \lm=-2, \lp=1$. Errors are rounded. }
 {\tiny
\begin{tabular}{c|c|cc|c|cc}
\hline\hline
$a_2/a_1$ & -0.075 & -0.05 & -0.025 & 0 & 0.025 \\\hline
0.025 & 0.0528532412024314& 0.0649856679446126 & 0.0879014169039586 & 0.506498701211731 & 0.923417160799492\\
0.05 & 0.0533971065051704 & 0.0656207900757623 & 0.0886699612390502 & 0.507497961893706 & 0.925278586629322\\
0.075 & 0.0536378889312988 & 0.0658957955144885 & 0.0889908892581356 & 0.507885843291178 & 0.925781540582068\\
0.1 & 0.0537738608706033 & 0.0660488001673687 & 0.0891656084917806 & 0.508089681056682 & 0.926027783268804\\
0.175 & 0.05396033997440315 & 0.0662551510091756 & 0.0893960371866518 & 0.508350135593746 & 0.926327268956837
\\\hline
\end{tabular}

\begin{tabular}{c|ccccc|ccccc}
\hline\hline
&&& $A$ & & && & B & \\\hline
$a_2/a_1$ & -0.075 & -0.05 & -0.025 & 0 & 0.025 & -0.075 & -0.05 & -0.025 & 0 & 0.025
 \\\hline
2.07E-05 &	2.35E-05 &	2.742E-05	 & 3.35E-05 &	5.11E-05 
&0.00039 &	0.00036 &	0.00031 &	6.61E-05 &	5.54E-05\\ 

0.05	& 1.17E-05 &	1.31E-05	& 1.48E-05 &	1.71E-05 &	2.06E-05
& 0.00022 &	0.00020 &	0.00017 & 	3.38E-05	& 2.23E-05\\

0.075 & 7.80E-06 &	8.59E-06 &	9.56E-06 &	1.08E-05 &	1.24E-05  &
0.00015 &	0.00013 &0.00011 &	2.12E-05	& 1.34E-05
\\
 
0.1 &	 5.56E-06	& 6.08E-06 &	6.693E-06 &	7.43E-06 &	8.36E-06 
& 0.00010&	9.21E-05 &	7.51E-05 &	1.46E-05 &	9.02E-06
\\

0.175 & 2.50E-06 &	2.69E-06 &	2.91E-06 &	3.16E-06 &	3.44E-06
& 4.64E-05 &	4.07E-05 &	3.263E-05 &	6.21E-06 &	3.72E-06
\end{tabular}
}
\begin{flushleft}{\tiny
Errors of the benchmark values in the continuous time model: better than E-14, at a number of points,  better than E-15. \\
A:  Errors of approximation of the continuous time model by the discrete time model, $\barDe=1/252$.\\
B: Relative rrors of approximation of the continuous time model by the discrete time model, $\barDe=1/252$.}
\end{flushleft}

\label{table:cpdf_disc_vs_cont0}
 \end{table}

\begin{table}
\caption{\small Joint cpdf $F(T,a_1,a_2):=\bQ[X_T\le a_1, \barX_T\le a_2\ |\ X_0=\barX_0=0]$, and errors (rounded)
and CPU time (in msec) of two numerical schemes. $T=5Y$. Discrete monitoring, the monitoring interval $\barDe=1/252$,
the number of time steps 1260. KoBoL 
close to the Variance Gamma model, with an almost symmetric jump density, and no ``drift": $m_2=0.1$, $\nu=0.2, \lm=-2, \lp=1$. Errors are rounded, the CPU time is in milliseconds  (average over 1000 runs).
 }
 {\tiny
\begin{tabular}{c|c|cc|c|cc}
\hline\hline
$a_2/a_1$ & -0.075 & -0.05 & -0.025 & 0 & 0.025 \\\hline
0.025 & 0.322715785176063 & 0.341705312612668 & 0.362654563514927 & 0.385503065295135 & 0.402853073943893\\
0.05 & 0.36823129960626 & 0.390755656346763 & 0.415922513339139 & 0.444104383367338 & 0.469731888892867\\
0.075 & 0.396209256972821 & 0.420842816392821 & 0.448475971976962 & 0.479643744322071 & 0.509135503443898\\
0.1 & 0.415752842072438 & 0.44180038793114 & 0.471059131572705 & 0.504139671693882 & 0.535967402412399\\

0.175 & 0.450253847495689 & 0.478623894580305 & 0.510496476100609 & 0.546559667768138 & 0.581857694138651
\\\hline
\end{tabular}

\begin{tabular}{c|ccccc|ccccc}
\hline\hline
&&& $A$ & & && & B & \\\hline
$a_2/a_1$ & -0.075 & -0.05 & -0.025 & 0 & 0.025 & -0.075 & -0.05 & -0.025 & 0 & 0.025
 \\\hline
0.025 &	5.26E-11 &	5.40E-11 &	5.56E-11 &	5.70E-11 &	6.97E-11

& 1.72E-11 &	1.86E-11 &	2.04E-11 &	2.58E-11 &	4.38E-11\\

0.05	& 5.19E-11 &	5.27E-11 &	5.35E-11 &	5.38E-11 &	5.45E-11

& 9.46E-12 &	1.02E-11 &	1.11E-11 &	1.54E-11 &	2.10E-11\\

0.075 & 5.44E-11 &	5.50E-11 &	5.55E-11 &	5.55E-11 & 5.56E-11
& 6.30E-12 &	6.77E-12 &	7.35E-12 &	1.12E-11 &	1.62E-11 \\
 
0.1 &	 5.74E-11 & 	5.78E-11 &	5.83E-11	& 5.80E-11 &	5.80E-11
& 4.80E-12 &	5.14E-12 &	5.58E-12 &	9.25E-12	& 1.40E-11
\\

0.175 & 6.63E-11 &	6.66E-11 &	6.69E-11 &	6.66E-11 &	6.63E-11

& 2.52E-12 &	2.68E-12 &	2.91E-12 &	6.35E-12 &	1.08E-11
\end{tabular}
}
\begin{flushleft}{\tiny
Errors of the benchmark values: better than E-14, at a number of points, better than E-15. CPU time per 1 point: 239, per 44 points: 2,019.\\
A:  Trapezoid rule, $M_0=2844$, $N^\pm=144$. CPU time per 1 point: 812.4; per 44 points:
9,211.\\
B: SINH applied to the inverse $Z$-transform, with $M_0=19, N^\pm=137$. CPU time per 1 point 15.7, per 44 points: 111.8.}
\end{flushleft}

\label{table:cpdf5}
 \end{table}
 
\begin{table}
\caption{\small Joint cpdf $F(T,a_1,a_2):=\bQ[X_T\le a_1, \barX_T\le a_2\ |\ X_0=\barX_0=0]$, in the continuous time model, and errors (rounded) of approximation by the discrete time model, with the time step $\barDe=1/252$. $T=5Y$. KoBoL 
close to the Variance Gamma model, with an almost symmetric jump density, and no ``drift": $m_2=0.1$, $\nu=0.2, \lm=-2, \lp=1$. Errors are rounded. }
 {\tiny
\begin{tabular}{c|c|cc|c|cc}
\hline\hline
$a_2/a_1$ & -0.075 & -0.05 & -0.025 & 0 & 0.025 \\\hline
0.025 & 0.322520199783594 & 0.341498771047289 & 0.362435915393477 & 0.385270783252495 & 0.402604045074505\\
0.05 & 0.368086705216705 & 0.390602983343428 & 0.415760934326279 & 0.443932843905084 & 0.469548901563007\\
0.075 & 0.396095198340728 & 0.420722495404712 & 0.448348786726241 & 0.479508955507981 & 0.50899215240333\\
0.1 & 0.415660121524435 & 0.441702681710183 & 0.470955990412814 &  0.50403055838718 & 0.535851652064689\\
0.175 & 0.45019937409531 & 0.47856666561315 & 0.510436280251142 & 0.546496264054607 & 0.581790803932632
\\\hline
\end{tabular}

\begin{tabular}{c|ccccc|ccccc}
\hline\hline
&&& $A$ & & && & B & \\\hline
$a_2/a_1$ & -0.075 & -0.05 & -0.025 & 0 & 0.025 & -0.075 & -0.05 & -0.025 & 0 & 0.025
 \\\hline
0.025 &0.00020 &	0.00021 &	0.00022 &	0.00023 & 	0.00025&
0.00061 &	0.00060 &	0.00060 &	0.00060 &	0.00062 \\ 

0.05	& 0.00015	& 0.00015 &	0.00016 &	0.00017 &	0.00018
& 0.00039&	0.00039 &	0.00039 &	0.00039 &	0.00039\\

0.075 & 0.00011 &	0.00012 &	0.00013 &	0.00013 &	0.00014
& 0.00029 &	0.00029 &	0.00028 &	0.00028 &	0.00028 \\
 
0.1 &	 9.27E-05 & 	9.77E-05 &	0.00010&	0.00011& 	0.00012& 
0.00022 &	0.00022 &	0.00022 &	0.00022 &	0.00022\\

0.175 & 5.45E-05	& 5.72E-05 &	6.02E-05 &	6.34E-05 &	6.69E-05

& 0.00012 &	0.00012 &	0.00012 &	0.00012 &	0.00011\end{tabular}
}
\begin{flushleft}{\tiny
Errors of the benchmark values in the continuous time model: better than E-15, with a couple of exceptions. \\
A:  Errors of approximation of the continuous time model by the discrete time model, $\barDe=1/252$.\\

B: Relative errors of approximation of the continuous time model by the discrete time model, $\barDe=1/252$.}
\end{flushleft}

\label{table:cpdf_disc_vs_cont5}
 \end{table}

\begin{table}
\caption{\small Joint cpdf $F(T,a_1,a_2):=\bQ[X_T\le a_1, \barX_T\le a_2\ |\ X_0=\barX_0=0]$, and errors (rounded)
and CPU time (in msec) of two numerical schemes. $T=15Y$. Discrete monitoring, the monitoring interval $\barDe=1/252$,
the number of time steps 3780. KoBoL 
close to the Variance Gamma model, with an almost symmetric jump density, and no ``drift": $m_2=0.1$, $\nu=0.2, \lm=-2, \lp=1$. Errors are rounded, the CPU time is in milliseconds  (average over 1000 runs).
 }
 {\tiny
\begin{tabular}{c|c|cc|c|cc}
\hline\hline
$a_2/a_1$ & -0.075 & -0.05 & -0.025 & 0 & 0.025 \\\hline
0.025 & 0.273003522656352 & 0.275060714621777 & 0.276863384237128 & 0.278361438403706 & 0.279413583881186\\
0.05 & 0.325601636899453 & 0.328403204232286 & 0.330932690547093 & 0.333148630324321 & 0.334989330744591\\
0.075 & 0.364467787376584 & 0.367935193185576 & 0.371127846709011 & 0.37400815325999 & 0.376529331567823\\
0.1 & 0.396164068347951 & 0.400244732717707 & 0.404054649236343 & 0.407558402431724 & 0.410715394955968\\
0.175 & 0.467032161225892 & 0.472690792930844 & 0.47810575026395 & 0.483245070441871 & 0.488075079451549
\\\hline
\end{tabular}

\begin{tabular}{c|ccccc|ccccc}
\hline\hline
&&& $A$ & & && & B & \\\hline
$a_2/a_1$ & -0.075 & -0.05 & -0.025 & 0 & 0.025 & -0.075 & -0.05 & -0.025 & 0 & 0.025
 \\\hline
0.025 &	1.40E-10 &	1.41E-10 &	1.41E-10	& 1.41E-10 &	1.45E-10 
& 4.50E-11 &	4.59E-11 &	4.67E-11 &	-9.51E-12	& 2.71E-11\\

0.05	& 1.56E-10 &	1.56E-10 &	1.56E-10	& 1.56E-10 &	1.56E-10
& 4.10E-11 &	4.18E-11 &	4.26E-11 &	-1.37E-11 &	-1.368E-11\\

0.075 & 1.70E-10 &	1.71E-10 &	1.719E-10	 & 1.71E-10 &	1.71E-10
& 4.05E-11 &	4.13E-11	& 4.219E-11 &	-1.42E-11 &	-1.41E-11
 \\
 
0.1 &	 1.84E-10 &	1.84E-10 &	1.84E-10 &	1.84E-10 &	1.84E-10
& 4.08E-11 &	4.16E-11 &	4.24E-11 &	-1.39E-11	& -1.38E-11
\\

0.175 & 2.18E-10 &	2.19E-10	& 2.19E-10 &	2.19E-10 &	2.18E-10
& 4.26E-11 &	4.34E-11 &	4.43E-11	& -1.120E-11 &	-1.19E-11
\end{tabular}
}
\begin{flushleft}{\tiny
Errors of the benchmark values: better than E-13, with a couple of exceptions. CPU time per 1 point: 548, per 44 points: 4,162.\\
A:  Trapezoid rule, $M_0=8538$, $N^\pm=172$. CPU time per 1 point: 2,494; per 44 points:
25,613.\\
NB: the general recommendation for the choice of $M_0$  (the error tolerance E-10) is decreased by  50\%.\\
B: SINH applied to the inverse $Z$-transform, with $M_0=65, N^\pm=144$. CPU time per 1 point 27.5, per 44 points: 319.9.}
\end{flushleft}

\label{table:cpdf2}
 \end{table}

\begin{table}
\caption{\small Joint cpdf $F(T,a_1,a_2):=\bQ[X_T\le a_1, \barX_T\le a_2\ |\ X_0=\barX_0=0]$, in the continuous time model, and errors (rounded) of approximation by the discrete time model, with the time step $\barDe=1/252$. $T=15Y$. KoBoL 
close to the Variance Gamma model, with an almost symmetric jump density, and no ``drift": $m_2=0.1$, $\nu=0.2, \lm=-2, \lp=1$. Errors are rounded. }
 {\tiny
\begin{tabular}{c|c|cc|c|cc}
\hline\hline
$a_2/a_1$ & -0.075 & -0.05 & -0.025 & 0 & 0.025 \\\hline
0.025 & 0.272804820564476 & 0.274859824870105 & 0.276660421125814 & 0.278156527723931 & 0.279206809800465\\
0.05 & 0.325434820584041 & 0.328234181455624 & 0.33076152710125 & 0.332975399214234 & 0.33481410584792\\
0.075 & 0.364320938273655 & 0.36778614523566 & 0.370976636656939 & 0.373854823172472 & 0.376373926428562\\

0.1 & 0.39603236911559 & 0.400110870613695 & 0.403918641220246 & 0.407420269437458 & 0.410575160782172\\

0.175 & 0.466932999301196 & 0.472589684334096 & 0.47800268074549 & 0.483140027197031 & 0.487968050938338
\\\hline
\end{tabular}

\begin{tabular}{c|ccccc|ccccc}
\hline\hline
&&& $A$ & & && & B & \\\hline
$a_2/a_1$ & -0.075 & -0.05 & -0.025 & 0 & 0.025 & -0.075 & -0.05 & -0.025 & 0 & 0.025
 \\\hline
0.025 &	0.00020 &	0.00020 &	0.00020 &	0.00021 &	0.00021 &
0.00073 &	0.00073 &	0.00073 &	0.00074	& 0.00074 \\ 

0.05	& 0.00017 &	0.00017 & 	0.00017 &	0.00017 &	0.00017 
& 0.00051 &	0.00052 & 	0.00052 &	0.00052 &	0.00052\\

0.075 & 0.00015 &	0.00015 &	0.00015 &	0.00015 &	0.00016
& 0.00040 &	0.00040 &	0.00041 &	0.00041 &	0.00041 \\
 
0.1 &	 0.00013 &	0.00013 &	0.00014 &	0.00014 &	0.00014
& 0.00033 &	0.00033 &	0.00034 & 	0.00034 &	0.00034\\

0.175 & 9.92E-05 &	0.00010 &	0.00010 &	0.00010 &	0.00011

& 0.00021&	0.00021 &	0.00022 &	0.00022 &	0.00022
\end{tabular}
}
\begin{flushleft}{\tiny
Errors of the benchmark values in the continuous time model: better than E-13, with a couple of exceptions. \\
A:  Errors of approximation of the continuous time model by the discrete time model, $\barDe=1/252$.\\

B: Relative errors of approximation of the continuous time model by the discrete time model, $\barDe=1/252$.}
\end{flushleft}

\label{table:cpdf_disc_vs_cont2}
 \end{table}

\begin{table}
\caption{\small Joint cpdf $F(T,a_1,a_2):=\bQ[X_T\le a_1, \barX_T\le a_2\ |\ X_0=\barX_0=0]$, and errors (rounded)
and CPU time (in msec) of two numerical schemes. $T=15Y$. Discrete monitoring, the monitoring interval $\barDe=1/252$,
the number of time steps 3780. KoBoL 
close to NIG, with an almost symmetric jump density, and no ``drift": $m_2=0.1$, $\nu=1.2, \lm=-2, \lp=1$. Errors are rounded, the CPU time is in milliseconds  (average over 1000 runs).
 }
 {\tiny
\begin{tabular}{c|c|cc|c|cc}
\hline\hline
$a_2/a_1$ & -0.075 & -0.05 & -0.025 & 0 & 0.025 \\\hline
0.025 & 0.08750889022257 & 0.0876433202115771 & 0.0877488959582886 & 0.0878234438630917 & 0.0878604203790796\\
0.05 & 0.133430426595114 & 0.133678790617469 & 0.133884632610215 & 0.134046292415956 & 0.13416044157208\\
0.075 & 0.172212077596399 & 0.172587459419214 & 0.172909022548126  & 0.173175531405955 & 0.173384836452991\\
0.1 & 0.206872444504732 & 0.207388307897551 & 0.207840301897249 & 0.208227492747064 & 0.208548393051015\\
0.175 & 0.295589651996839 & 0.296599359388506 & 0.297519605986825 & 0.298349844042413 & 0.299089406243993
\\\hline
\end{tabular}

\begin{tabular}{c|ccccc|ccccc}
\hline\hline
&&& $A$ & & && & B & \\\hline
$a_2/a_1$ & -0.075 & -0.05 & -0.025 & 0 & 0.025 & -0.075 & -0.05 & -0.025 & 0 & 0.025
 \\\hline
0.025 &	-3.95E-10	& -4.66E-10 &	-5.44E-10 &	1.06E-09	& 4.637E-10

& 4.68E-11 &	4.69E-11 &	4.16E-11 &	-6.00E-11 &	-5.90E-11\\

0.05	& -5.48E-10 &	-6.63E-10 &	-7.95E-10 & 	7.44E-10 &	7.73E-10

& 5.70E-11 &	5.94E-11 &	5.57E-11 &	-4.62E-11	& -8.14E-11\\

0.075 & -6.54E-10 &	-8.08E-10 & 	-9.90E-10 & 	4.86E-10 &	4.33E-10

& 6.063E-11 &	6.49E-11 &	6.33E-11 &	-3.66E-11	& -7.06E-11
 \\
 
0.1 &	 1.84E-10 &	1.84E-10 &	1.84E-10 &	1.84E-10 &	1.84E-10

&6.09E-11 &	6.62E-11 &	6.61E-11 &	-3.21E-11 & 	-6.42E-11
\\

0.175 & -6.65E-10 &	-8.321E-10 &	-1.03E-09	& 4.21E-10 &	3.37E-10

& 5.80E-11 &	6.34E-11 &	6.38E-11 &	-3.32E-11 & 	-6.34E-11
\end{tabular}
}
\begin{flushleft}{\tiny
Errors of the benchmark values: better than $5\cdot 10^{-13}$. CPU time per 1 point: 1,848, per 44 points: 20,263.\\
A:  Trapezoid rule, $M_0=8538$, $N^\pm=172$. CPU time per 1 point: 3,046; per 44 points:
35,481.\\
NB: the general recommendation for the choice of $M_0$  (the error tolerance E-10) is decreased by  47\%.\\
B: SINH applied to the inverse $Z$-transform, with $M_0=28, N^\pm=183$. CPU time per 1 point 27.5, per 44 points: 219.9.}
\end{flushleft}

\label{table:cpdf3}
 \end{table}
 
 \begin{table}
\caption{\small Joint cpdf $F(T,a_1,a_2):=\bQ[X_T\le a_1, \barX_T\le a_2\ |\ X_0=\barX_0=0]$, in the continuous time model, and errors (rounded) of approximation by the discrete time model, with the time step $\barDe=1/252$. $T=15Y$. KoBoL 
close to NIG, with an almost symmetric jump density, and no ``drift": $m_2=0.1$, $\nu=1.2, \lm=-2, \lp=1$. Errors are rounded. }
 {\tiny
\begin{tabular}{c|c|cc|c|cc}
\hline\hline
$a_2/a_1$ & -0.075 & -0.05 & -0.025 & 0 & 0.025 \\\hline

0.025 &  0.083599231183863 & 0.083725522194071 & 0.0838241629685378 & 0.0838929695457668 & 0.0839249287233805\\
0.05 & 0.130217710987261 & 0.130456839782095 & 0.130654399607263 &  0.130808705570046 & 0.13091634106018\\
0.075 & 0.169363038877019 & 0.169728032397852 & 0.170040043384744 & 0.170297815998657 & 0.170499151715123\\

0.1 & 0.204270598983103 & 0.204774983963964 & 0.205216260776888 & 0.205593481299844 & 0.205905127535884\\

0.175 & 0.293472724302235 & 0.294468206269081 & 0.295374834640356 & 0.296192060853885 & 0.296919211526691
\\\hline
\end{tabular}

\begin{tabular}{c|ccccc|ccccc}
\hline\hline
&&& $A$ & & && & B & \\\hline
$a_2/a_1$ & -0.075 & -0.05 & -0.025 & 0 & 0.025 & -0.075 & -0.05 & -0.025 & 0 & 0.025
 \\\hline
0.025 &	0.0039 &	0.0039 &	0.0039 &	0.0039 &	0.0039 &

0.047 &	0.047 &	0.047 &	0.047 &	0.047 \\ 

0.05	& 0.0032&	0.0032 &	0.0032 &	0.0032 &	0.0032 

& 0.025 &	0.025 &	0.025 &	0.025 &	0.025\\

0.075 & 0.0028 &	0.0029 &	0.0029 &	0.0029 &	0.0029

& 0.017 &	 0.017 &	 0.017 &	 0.017 &	 0.017 \\
 
0.1 &	  0.013 &	0.013 &	0.013 &	0.013&	0.013

& 0.00033 &	0.00033 &	0.00034 & 	0.00034 &	0.00034\\

0.175 & 0.0021 &	0.0021 &	0.0021 &	0.0022 &	0.0022

& 0.0072 &	0.00721 &	0.0073 &	0.0073 &	0.0073
\end{tabular}
}
\begin{flushleft}{\tiny
Errors of the benchmark values in the continuous time model: better than E-13, with a couple of exceptions. \\
A:  Errors of approximation of the continuous time model by the discrete time model, $\barDe=1/252$.\\

B: Relative errors of approximation of the continuous time model by the discrete time model, $\barDe=1/252$.}
\end{flushleft}

\label{table:cpdf_disc_vs_cont3}
 \end{table}

 \end{document}